\documentclass{amsart}

\usepackage{amssymb,bbm,xcolor,graphicx,overpic}

\newtheorem{theorem}{Theorem}[section]
\newtheorem{corollary}{Corollary}[theorem]
\newtheorem{lemma}[theorem]{Lemma}
\renewcommand{\vec}{\textbf}
\theoremstyle{definition}
\newtheorem{definition}[theorem]{Definition}

\theoremstyle{remark}
\newtheorem{remark}[theorem]{Remark}
\newcommand{\D}{\mathrm{d}}
\DeclareMathOperator{\E}{e}
\DeclareMathOperator{\I}{i}
\DeclareMathOperator{\tr}{tr}
\newcounter{alg}

\numberwithin{equation}{section}

\begin{document}

\title[The CGA is almost deterministic]{The conjugate gradient algorithm on well-conditioned Wishart matrices is almost deterministic}


\author{Percy Deift}
\address{New York University\\
Courant Institute of Mathematical Sciences\\
251 Mercer St.\\
New York, NY 10012}
\curraddr{}
\email{deift@cims.nyu.edu}
\thanks{}

\author{Thomas Trogdon}
\address{University of Washington\\
Department of Applied Mathematics\\
Seattle, WA 98195-3925}
\curraddr{}
\email{trogdon@uw.edu}
\thanks{We are grateful for discussions with Elliot Paquette, Joel Tropp and Roman Vershynin that have greatly improved the paper.  This work was supported in part by NSF    DMS-1300965 (PD) and NSF DMS-1753185, DMS-1945652 (TT)}
\newcommand{\noneed}[1]{{\color{red}{#1}}}
\newcommand{\addit}[1]{{\color{blue}{#1}}}

\date{}

\dedicatory{}

\begin{abstract}
We prove that the number of iterations required to solve a random positive definite linear system with the conjugate gradient algorithm  is almost deterministic for large matrices.  We treat the case of Wishart matrices $W = XX^*$ where $X$ is $n \times m$ and $n/m \sim d$ for $0 < d < 1$.  Precisely, we prove that for most choices of error tolerance, as the matrix increases in size, the probability that the iteration count deviates from an explicit deterministic value tends to zero.  In addition, for a fixed iteration count, we show that the norm of the error vector and the norm of the residual converge exponentially fast in probability, converge in mean and converge almost surely.  
\end{abstract}

\maketitle

\section{Introduction}

The conjugate gradient algorithm (CGA) \cite{Hestenes1952} is arguably the most effective iterative method from numerical linear algebra.    In exact arithmetic, the algorithm requires at most $n$ iterations to solve a $n \times n$ positive-definite linear system and it often requires many less iterations to compute a good approximate solution.  It is exceedingly simple to implement and there are well-known error bounds available. And, despite the fact that the CGA is sensitive to round-off errors these error bounds still effectively hold for floating point arithmetic \cite{Greenbaum1989}.  While we present the algorithm in full below (see Algorithm~\ref{a:cga}), the variational characterization of the method is summarized as follows:  Consider the linear system $W \vec x = \vec b$, $W>0$. Given an initial guess $\vec x_0$, find the unique vector $\vec x_k$ that satisfies
\begin{align*}
    &\|\vec x - \vec x_k\|_W = \min_{\vec y \in \mathcal X_k} \|\vec y - \vec x\|_W,\\
&\mathcal X_k= \vec x_0 + \mathrm{span}\{\vec r_0, W \vec r_0, \ldots, W^{k-1}\vec r_0\},\quad \|\vec y \|_W^2 = \vec y^*W\vec y, \quad \vec r_0 = \vec b - W \vec x_0.
\end{align*}
At each step $k$ of the iteration one can easily construct $\vec x_k$ and the algorithm itself computes $\vec r_k = \vec b - W \vec x_k$, $k = 0,1,2,\ldots,n$.  One has to then determine a computable stopping criterion, and typically, the algorithm is halted when $\|\vec r_k\|_{2} < \epsilon$, $\|\vec y\|^2_{2} = \vec y^* \vec y$, for a chosen error tolerance $\epsilon.$

Here we focus on two main measures of the error,  $\vec e_k(W,\vec b) = \vec e_k : = \vec x - \vec x_k$:
\begin{align*}
    \| \vec e_k \|_W \quad \text{and}\quad  \|\vec r_k \|_{2} = \| \vec e_k \|_{W^2}, \quad \vec r_k(W,\vec b) = \vec r_k = \vec b - W \vec x_k,
\end{align*}
in the particular case when $\vec x_0 = 0$.  The associated halting times are
\begin{align}\label{eq:htimes}
\begin{split}
     t^{(1)}_{\epsilon}(W,\vec b) &= \min\{ k: \|\vec e_k\|_{W} < \epsilon \},\\
    t^{(2)}_{\epsilon}(W,\vec b) &= \min\{ k: \|\vec r_k\|_2 < \epsilon \}.
\end{split}
\end{align}
We emphasize the importance of analyzing both quantities because $\vec r_k$ is what is observed throughout the iteration and, of course, $\vec e_k$ is the true error.  

Our results (Theorems~\ref{t:estimates}, \ref{t:ek}, \ref{t:halting}) are derived for both real and complex Gaussian matrices \footnote{We can also easily extend the results to the case of quarternion entries.}.  We assume 
\begin{align}\label{eq:W}
W = XX^*/m,
\end{align}
where $X$ is an $n \times m$ matrix whose entries are iid real or complex standard normal random variables. This is the real or complex Wishart distribution. Suppose further that $m = \lfloor n/d \rfloor$ for $0 < d \leq 1$ (Note that if $d > 1$, i.e. $m < n$, then $W$ is singular and $W \vec x = \vec b$ does not have a unique solution).  Then if $\vec b$ is a random unit vector, independent of $W$, our results show that as $n \to \infty$
\begin{align*}
    \| \vec r_k (W,\vec b)\|_{2} \overset{\text{almost surely}}{\longrightarrow} d^{k/2}.
\end{align*}
If $ d < 1$ then as $n \to \infty$
\begin{align*}
 \| \vec e_k(W,\vec b) \|_{W} \overset{\text{almost surely}}{\longrightarrow} \frac{d^{k/2}}{\sqrt{1-d}}.
\end{align*}
Furthermore, there are discrete sets $S_d^{(1)}$ and $S_d^{(2)}$ with the property that if $\epsilon >0$ is in the complement of these sets, $\epsilon$ fixed, then 
\begin{align*}
 \lim_{n \to \infty} &\mathbb P \left(  t^{(1)}_{\epsilon}(W,\vec b) = \left\lceil \frac{2\log \epsilon + \log (1-d)}{\log d} \right\rceil \right) = 1, \quad \epsilon < (1-d)^{-1}, \quad x \not\in S_d^{(1)},\\
   \lim_{n \to \infty} &\mathbb P \left(  t^{(2)}_{\epsilon}(W,\vec b) = \left\lceil \frac{2\log \epsilon}{\log d} \right\rceil  \right) = 1, \quad \epsilon < 1, \quad x \not \in S_d^{(2)}.
\end{align*}
Therefore, the halting time becomes effectively deterministic.  We also present estimates that demonstrate that the probability that the errors $\vec e_k$ deviate from their means decays exponentially with respect $n$.  In the case $d = 1$, a consequence of our results is that for any fixed $k >0$ and $\epsilon < 1$,
\begin{align*}
    \lim_{n \to \infty} \mathbb P( t^{(2)}_{\epsilon}(W,\vec b) > k) = 1.
\end{align*}

\begin{remark}
It is important to point out that $W$ in \eqref{eq:W} is not necessarily a near-identity matrix.  Indeed as $n \to \infty$, the eigenvalues of $W$ typically lie in the interval
\begin{align*}
    \left[(1-\sqrt{d})^2,(1 + \sqrt{d})^2\right],
\end{align*}
and have an asymptotic density given by the famous Marchenko--Pastur law, see Definition~\ref{d:MP}.  For finite $n$, some of the eigenvalues of $W$ lie outside this interval, and the control of these eigenvalues plays a crucial role in the proofs of Theorems 3.1, 3.2 and 3.3 (see, for example, the proofs of \eqref{eq:largest} and \eqref{eq:smallest}).
\end{remark}

    Our proofs make critical use of the invariance of the Wishart distribution and the relation between Householder bidiagonalization and the Lanczos iteration.  This allows one to use classical estimates on chi-distributed random variables in a crucial way.  The specific tools and results we incorporate from  random matrix theory include global eigenvalue estimates \cite{Davidson2001}, the convergence of the empirical spectral measure \cite{Bai2007} and the central limit theorem for linear statistics \cite{Lytova2009}.   

The remainder of the paper is setup as follows.  In Section~\ref{sec:comp} we compare our analysis with facts already known about the conjugate gradient algorithm.  We also demonstrate our results with  numerical examples.  In Section~\ref{sec:bidiag} we introduce our random matrix ensembles, the basic definitions from random matrix theory and review the Householder bidiagonalization procedure applied to these ensembles.  We also review the connections between the conjugate gradient algorithm, the Lanczos iteration and the Householder bidiagonalization procedure.  In Section~\ref{sec:main} we present our main theorems.  In Section~\ref{sec:tech} we introduce the results from probability and random matrix theory that are required to prove our theorems.  In Section~\ref{sec:proof}  we give the proofs of the theorems.

\subsection{Comparison and demonstration}\label{sec:comp}

We now give a demonstration and discussion of the results.  In what follows $\langle \cdot \rangle$ denotes the sample average of a random variable using $20,000$ samples.  We will refer to the matrix
\begin{align*}
    W = XX^*/m
\end{align*}
where $X$ is an $n \times m$ matrix, having iid entries, $X_{11} = \pm 1$ with equal probability, as the \emph{Bernoulli ensemble} (BE).
\subsubsection{A numerical demonstration}

To demonstrate our main results,  in Figure~\ref{f:esamp} we plot the following quantities as a function of $k$ for different values of $n$
\begin{align*}
    \langle \|\vec e_k(W,\vec b)\|_W \rangle \quad \text{and} \quad \frac{d^{k/2}}{\sqrt{1-d}}
\end{align*}
with error bars that indicate where $99.9$\% of the samples lie.  In Figure~\ref{f:rsamp} we plot the same statistics for
\begin{align*}
    \langle \|\vec r_k(W,\vec b)\|_{2} \rangle \quad \text{compared with} \quad d^{k/2}.
\end{align*}
Both Figure~\ref{f:esamp} and Figure~\ref{f:rsamp} demonstrate the concentration of the errors about their means.
We demonstrate the limiting behavior of the halting times $t_\epsilon^{(j)}$ in Figure~\ref{f:halting}.  

In all of these figures we have included computations for distributions of random matrices, in particular the Bernoulli ensemble, that are beyond the class for which our results apply.  Nonetheless, it is clear that the behavior persists.  This universality will be investigated in future work.
\begin{figure}[tbp]
\begin{overpic}[width=.9\linewidth]{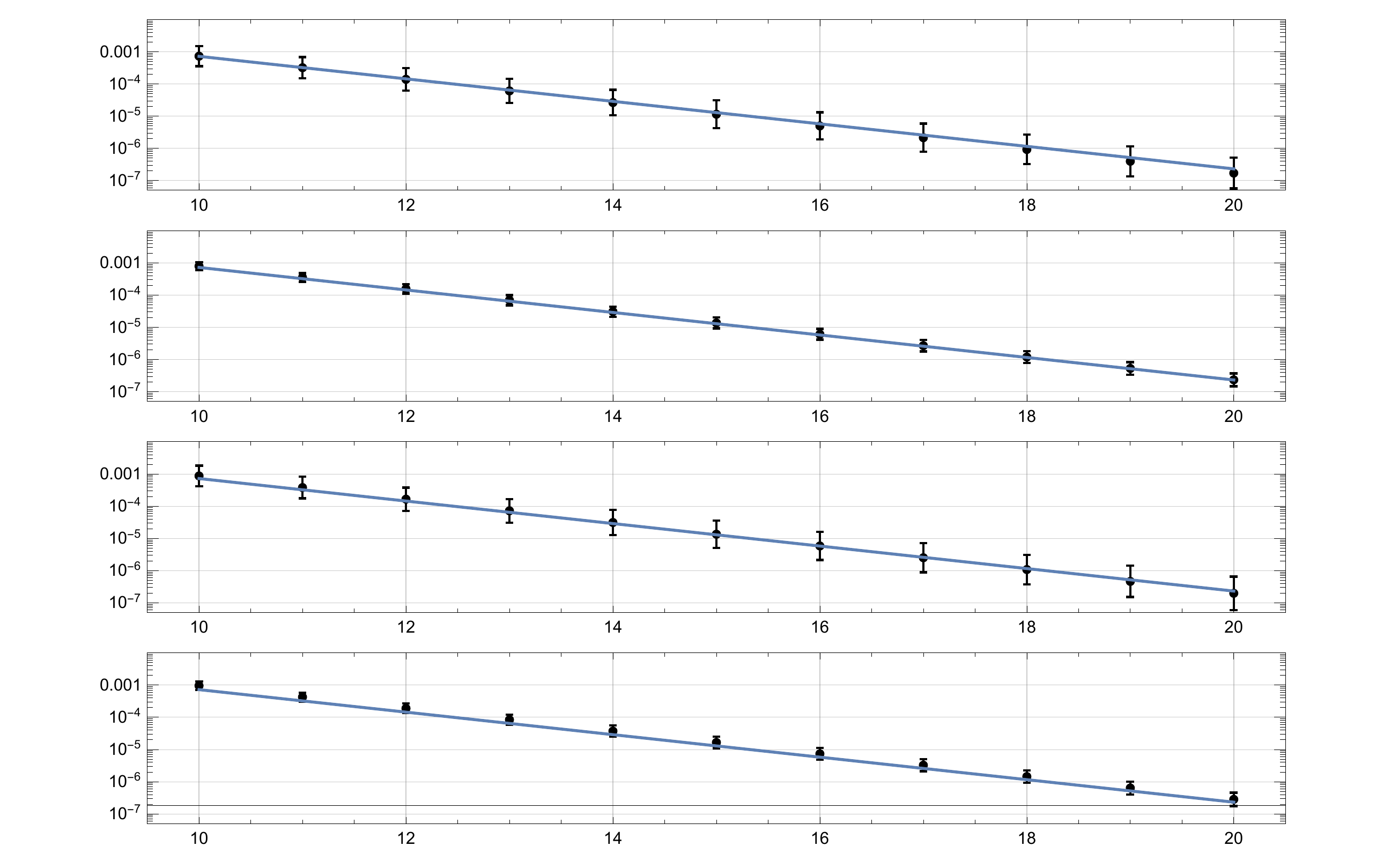}
\put(71,57){$\beta = 1$,~$n = 200$}
\put(71,42){$\beta = 1$,~$n = 800$}
\put(71,27){BE,~$n = 200$}
\put(71,12){BE,~$n = 800$}
\put(43,-4){$k$}
\put(2,26){\rotatebox{90}{$\|\vec e_k\|_{W}$}}
\end{overpic}
\vspace{.1in}
\caption{\label{f:esamp} A demonstration that $\|\vec e_k\|_{W}$ concentrates strongly around is mean, which is nearly equal to $d^{k/2}/\sqrt{1-d}$ (solid).  This plot is for $d = 0.2$.  The dots give the sample mean over $20,000$ samples and the error bars give the symmetric interval where $99.9\%$ of the samples lie.  It is clear that this interval shrinks rapidly as $n$ increases.  }
\end{figure}
\begin{figure}[tbp]
\begin{overpic}[width=.9\linewidth]{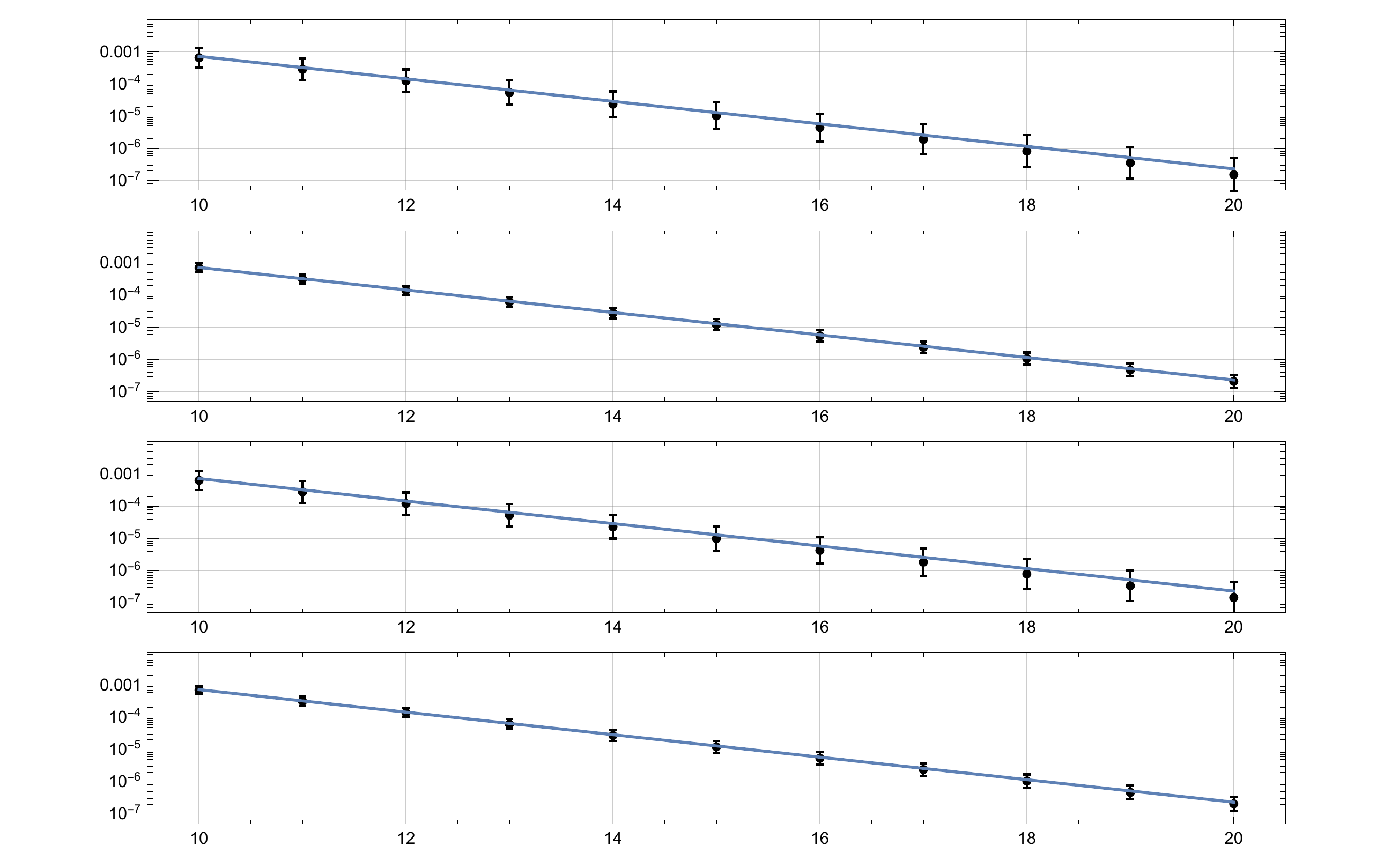}
\put(71,57){$\beta = 1$,~$n = 200$}
\put(71,42){$\beta = 1$,~$n = 800$}
\put(71,27){BE,~$n = 200$}
\put(71,12){BE,~$n = 800$}
\put(43,-4){$k$}
\put(2,26){\rotatebox{90}{$\|\vec r_k\|_{2}$}}
\end{overpic}
\vspace{.1in}
\caption{\label{f:rsamp} A demonstration that $\|\vec r_k\|_{2}$ concentrates strongly around is mean, which is nearly equal to $d^{k/2}$ (solid).  This plot is for $d = 0.2$.  The dots give the sample mean over $20,000$ samples and the error bars give the symmetric interval where $99.9\%$ of the samples lie.  It is clear that this interval shrinks rapidly as $n$ increases.  }
\end{figure}

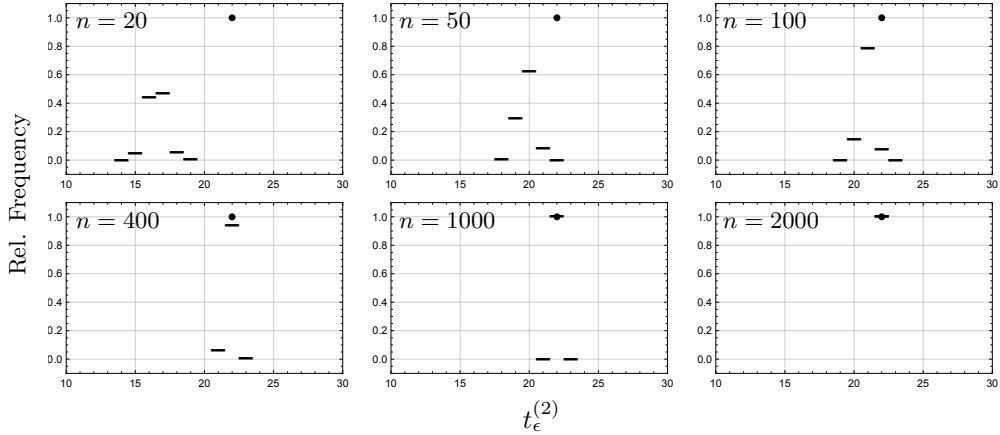
\begin{figure}[tbp]
\begin{overpic}[width=\linewidth]{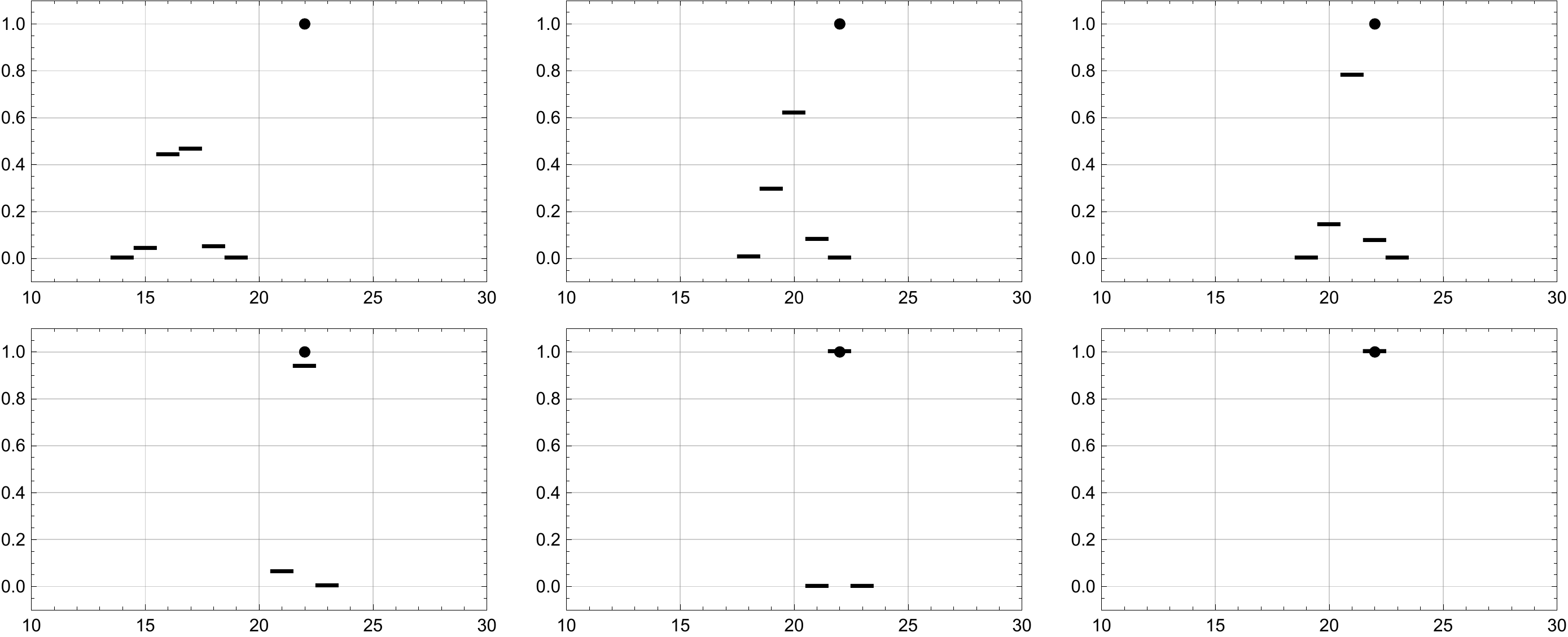}
\put(3,38){\small $n = 20$}
\put(37,38){\small $n = 50$}
\put(71,38){\small $n = 100$}
\put(3,17){\small $n = 400$}
\put(37,17){\small $n = 1000$}
\put(71,17){\small $n = 2000$}
\put(50,-4){$t_\epsilon^{(2)}$}
\put(-4,12){\rotatebox{90}{\small Rel. Frequency}}
\end{overpic}
\vspace{.1in}
\caption{\label{f:halting} A demonstration that $t_{\epsilon}^{(2)}$ becomes almost deterministic as $n$ increases for $d = 0.2$, $\epsilon = 6.627 \times 10^{-8}$ and  $\beta = 1$ using $20,000$ samples. Each panel gives the empirical halting time distribution for the indicated value of $n$.  For $n$ sufficiently large, $n \geq 2000$, the histogram is extremely concentrated.  }
\end{figure}
\subsubsection{Relation to previous work}

The classical error estimate for the CGA is \cite{Hestenes1952,Greenbaum1989}
\begin{align}\label{eq:CGest}
    \| \vec e_k(W,\vec b) \|_W \leq 2 \left[ \left(\frac{\sqrt{\kappa} -1}{\sqrt{\kappa}+1}\right)^k + \left(\frac{\sqrt{\kappa} -1}{\sqrt{\kappa}+1}\right)^{-k}\right]^{-1}\| \vec e_0(W,\vec b) \|_W,
\end{align}
where $\kappa = \lambda_1/\lambda_n$ is the condition number of $W$.  Here $\lambda_1 \geq \cdots \geq \lambda_n > 0$ are the eigenvalues of $W$.  It is a classical result in random matrix theory \cite{Bai1993} that the condition number of \eqref{eq:W} converges almost surely to $\frac{(1 + \sqrt{d})^2}{(1 - \sqrt{d})^2}$.  Roughly, one then obtains
\begin{align*}
     \| \vec e_k(W,\vec b) \|_W \lesssim 2 \left[ d^{k/2} + d^{-k/2} \right]^{-1}\| \vec e_0(W,\vec b) \|_W,
\end{align*}
which is often just simplified to
\begin{align*}
    \| \vec e_k(W,\vec b) \|_W \lesssim 2 d^{k/2}\| \vec e_0(W,\vec b) \|_W.
\end{align*}
This overestimates the actual error by just a factor of 2.

In \cite{Menon2016a}, the authors used \eqref{eq:CGest} and tail bounds on the condition number to estimate the halting times \eqref{eq:htimes} in the case $d = 1 + o(1)$.  A key observation was that the actual number of iterations appears to be of the same asymptotic order as the estimate obtained using \eqref{eq:htimes}.  This is something that will indeed be true if the error estimate used decays exponentially and turns out to be an overestimate by a constant factor.

\begin{remark}
Of particular interest is this case where $d$ depends on $n$ and $d \to 1$ as $n \to \infty$.  For example, $d = 1 - 1/n^{-1/2}$ was seen in \cite{Deift2015,Deift2014} to produce universal fluctuations for the halting times.  Similarly, one would want to treat the case $\epsilon = \epsilon(n) \to 0$ as $n \to \infty.$
\end{remark}

\begin{remark}
Our calculations in this work apply only to matrices with Gaussian entries.  An important question, one of universality, is if our results hold when this assumption is relaxed.  Indeed, one expects this to be true by the computations in Figures~\ref{f:esamp}, \ref{f:rsamp} and \ref{f:clt} and the wealth of theoretical universality results from random matrix theory \cite{Pillai2014,Bloemendal2016,Bai2007}.
\end{remark}

\section{The bidiagonalization of Wishart matrices and invariance}\label{sec:bidiag}

\begin{definition}\label{d:MP}
For $0 < d \leq 1$ set $m = \lfloor n/d \rfloor $. Let $X$ be an $n \times m$ matrix of iid standard normal random variables ($\beta = 1$) or $X = X_1 + \I X_2$  where $X_1$ and $X_2$ are independent copies of an $n \times m$ matrix of iid standard normal random variables ($\beta = 2$).  Then
\begin{align}
    W_{n,\beta,d} := \frac{1}{\beta m} XX^*
\end{align}
has the $\beta$-Wishart distribution.   The associated empirical spectral measure (ESM) is given by
\begin{align*}
    \mu_{n,\beta,d} = \frac{1}{n} \sum_{j=1}^n \delta_{\lambda_j(n,\beta,d)}
\end{align*}
where
$\lambda_1(n,\beta,d) \geq \lambda_2(n,\beta,d) \geq \cdots \geq \lambda_ n(n,\beta,d)$ are the eigenvalues of $W_{n,\beta,d}$.
Define the averaged EMS $\mathbb E \mu_{n,\beta}$ (or density of states) by 
\begin{align}\label{eq:dos}
    \int f(\lambda) \mathbb E \mu_{n,\beta,d}( \D \lambda) : = \mathbb E \left( \int f(\lambda) \mu_{n,\beta,d}(\D \lambda) \right)
\end{align}
for every\footnote{$C_b(\mathbb R^+)$ denotes bounded continuous functions on $[0,\infty)$.} $f \in C_b(\mathbb R^+)$. 
\end{definition}

\begin{definition}
Marchenko--Pastur law $\mu_{\mathrm{MP},d}$ on $\mathbb R$ is given by the density
\begin{align*}
    \rho_{\mathrm{MP,d}}(x) = \frac{1}{2 \pi d} \frac{\sqrt{|(d_+ - x)(x-d_-)|}}{x} \mathbbm{1}_{[d_-,d_+]}(x), \quad d_\pm = (1 \pm \sqrt{d})^2.
\end{align*}
\end{definition}
The relation of the Marchenko--Pastur law to the eigenvalues of a Wishart matrix is given in the following section.  But we demonstrate this relationship in Figure~\ref{f:mpdisplay}.

\begin{figure}[tbp]
\begin{overpic}[width=.9\linewidth]{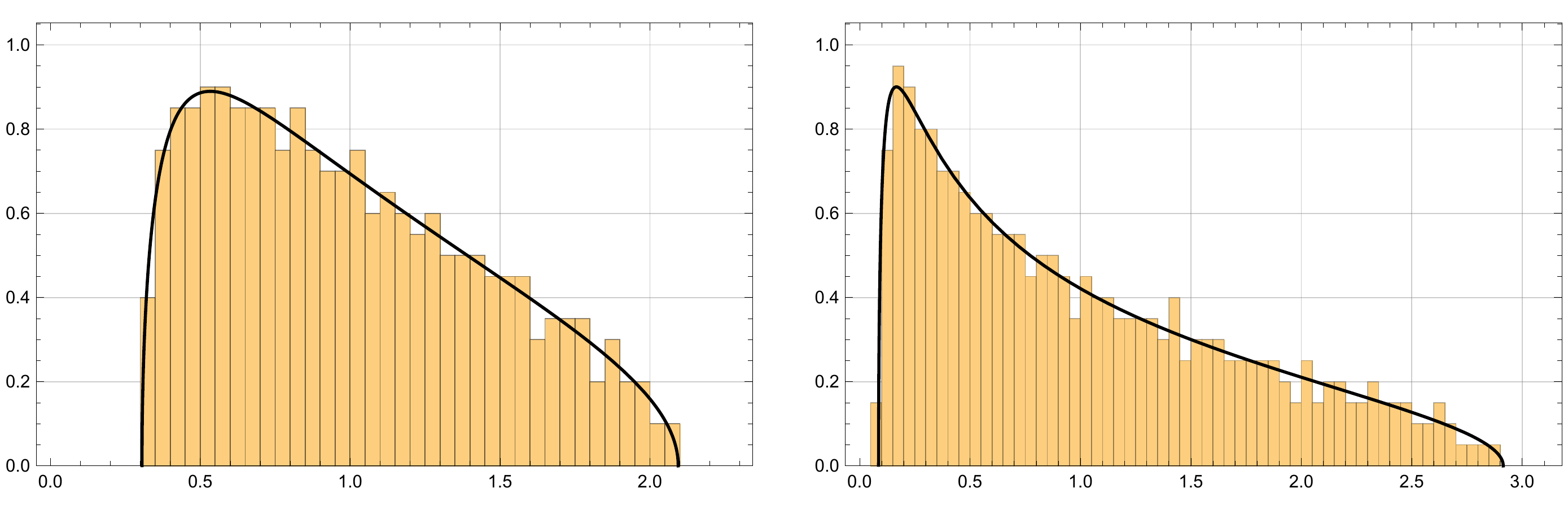}
\put(37,27){$d = 0.2$}
\put(88,27){$d = 0.5$}
\put(68,-2){Eigenvalues}
\put(17,-2){Eigenvalues}
\put(-3,4){\rotatebox{90}{Relative Frequency}}
\end{overpic}
\caption{\label{f:mpdisplay} A demonstration of how the Marchenko--Pastur law relates to the spectrum of a Wishart matrix for $n = 400$.  A histogram for the eigenvalues of one sampled matrix closely approximates the Marchenko--Pastur density.}

\end{figure}
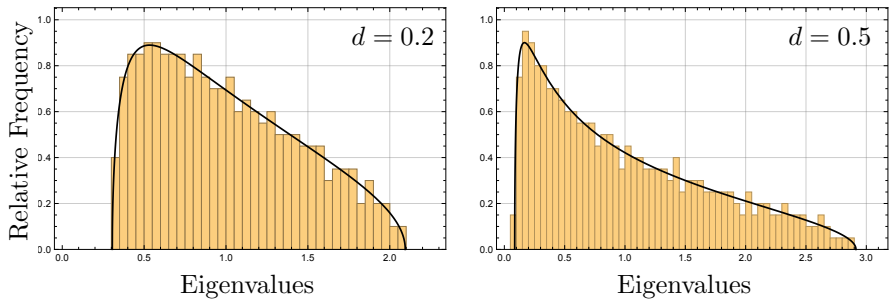

The Wishart distribution is invariant under orthogonal ($\beta = 1$) or unitary ($\beta =2$) conjugation.  Using $\beta = 2$, this means that if $U$ is a random unitary matrix that is independent of $W_{n,\beta,d}$ then
\begin{align*}
    U W_{n,\beta,d} U^* \overset{\text{dist.}}{=} W_{n,\beta,d}.
\end{align*}
For $W_{n,\beta,d} = \frac{1}{\beta m} XX^*$ the Householder bidiagionalization procedure \cite{TrefethenBau} operates on $X$ on the left and the right with $n \times n$ Householder reflections $R_1,R_2,\ldots,R_n$, and $m \times m$ Householder reflections $\tilde R_1,R_2,\ldots,\tilde R_n$,  so that
\begin{align}\label{eq:H}
\begin{bmatrix} H_{n,\beta,d} & 0 \end{bmatrix} : &= R_nR_{n-1}\cdots R_1 X \tilde R_1 \tilde  R_2 \cdot \tilde R_{n} \\
&= \begin{bmatrix} 
 \zeta_{11} & 0 & \cdots &&&& 0 \\
 \zeta_{21} & \zeta_{22}& 0 & \cdots &&& 0\\
 & \ddots & \ddots &   & && \\
 & & \zeta_{n,n-1} & \zeta_{nn}& 0 & \cdots & 0
\end{bmatrix}
\end{align}
where all entries are non-negative.  Because of invariance, $\{\zeta_{ij}\}$ are independent $\chi$-distributed random variables, see \cite{Dumitriu2002} and the references therein.  Specifically,
\begin{align}\label{eq:Hnb}
    H_{n,\beta,d} \overset{\text{dist.}}{=} \begin{bmatrix} \chi_{\beta m} & \\
    \chi_{\beta(n-1)} & \chi_{\beta(m-1)}\\
    & \chi_{\beta(n-2)} & \chi_{\beta(m-2)}\\
     && \ddots & \ddots \\
     &&& \chi_{\beta} & \chi_{\beta(m-n+1)}
    \end{bmatrix}
\end{align}
where all entries are independent.  Define the infinite matrix $\mathbb T_d$ by the entry-wise limit
\begin{align*}
    (\mathbb T_{d})_{ij}  := \lim_{n \to \infty} \frac{1}{\beta m}\mathbb E\left[ \left(H_{n,\beta,d}H_{n,\beta,d}^* \right)_{ij} \right], \quad 1 \leq i,j.
\end{align*}
Therefore
\begin{align*}
    \mathbb T_d &=  \mathbb H_{d} \mathbb H_{d}^*,\\
    \mathbb H_{d} &= \begin{bmatrix} 1 & \\
   \sqrt{d} & 1\\
    & \sqrt{d} & 1\\
     && \ddots & \ddots
    \end{bmatrix}.
\end{align*}
Lastly, define $\mathbb T_{k,d}$ to be the upper-left $k\times k$ submatrix of $\mathbb T_d$.

\subsection{Householder bidiagonalization, the Lanczos iteration and the CG algorithm}

The conjugate gradient algorithm (CGA) for the iterative solution of
\begin{align*}
    W \vec x = \vec b, \quad W > 0
\end{align*}
is given by\\

\noindent\fbox{%
\refstepcounter{alg}
    \parbox{\textwidth}{%
\flushright \boxed{\text{Algorithm~\arabic{alg}: Conjugate Gradient Algorithm\label{a:cga}}}
\begin{enumerate}
    \item $\vec x_0$ is the initial guess.
    \item Set $\vec r_0 = \vec b - W \vec x_0$, $\vec p_0= \vec r_0$.
    \item For $k = 1,2,\ldots,n$
    \begin{enumerate}
        \item Compute $\displaystyle a_{k-1} = \frac{\vec r_{k-1}^* \vec r_{k-1}}{\vec r_{k-1}^* W \vec p_{k-1}}$.
        \item Set $\vec x_{k} = \vec x_{k-1} + a_{k-1} \vec p_{k-1}$.
        \item Set $\vec r_{k} = \vec r_{k-1} - a_{k-1} W \vec p_{k-1}$.
        \item Compute $\displaystyle b_{k-1} = - \frac{\vec r_k^* \vec r_k}{\vec r_{k-1}^* \vec r_{k-1}}$.
        \item Set $\vec p_k =  \vec r_k - b_{k-1} \vec p_{k-1}$.
    \end{enumerate}
\end{enumerate}
    }%
}
\vspace{.05in}

The error at step $k$ is given by $\vec e_k = \vec x - \vec x_k$.  Define the norm $\|\vec y\|_W^2 = \vec e_k^* W \vec e_k$.  A variational characterization of the CGA is that
\begin{align*}
  \|\vec e_k\|_W = \|p_k^\dagger(W)\vec e_0\|_W = \min_{p_k \in \mathbb P_k^{(0)}} \|p_k(W)\vec e_0\|_W,
\end{align*}
where $\mathbb P_k^{(0)} = \{ p : p ~\text{ is a polynomial of degree $k$}, ~ p(0) = 1\}$.  The unique minimizer $p_k^\dagger$ in $\mathbb P_k^{(0)}$ can be described by the Lanczos algorithm.  The Lanczos algorithm is a tridiagonalization algorithm given by\\

\noindent\fbox{%
\refstepcounter{alg}
    \parbox{\textwidth}{%
\flushright \boxed{\text{Algorithm~\arabic{alg}: Lanczos Iteration\label{a:lanczos}}}
\begin{enumerate}
    \item $\vec y_1$ is the initial vector.  Suppose $\|\vec y_1\|_2^2 = \vec y_1^* \vec y_1 = 1$
    \item Set $\beta_0 = 0$
    \item For $k = 1,2,\ldots,n$
    \begin{enumerate}
        \item Compute $\displaystyle \alpha_{k} = (W \vec y_k - \beta_{k-1} \vec y_{k-1})^* \vec y_k$.
        \item Set $\vec v_k = W \vec y_k - \alpha_k \vec y_k - \beta_{k-1} \vec y_{k-1}$.
        \item Compute $\beta_k = \|\vec v_k\|_2$ and if $\beta_k \neq 0$, set $\vec y_{k+1} = \vec v_k/\beta_k$.
    \end{enumerate}
\end{enumerate}
    }%
}
\vspace{.05in}

The Lanczos algorithm produces a tridiagonal matrix $T$
\begin{align*}
    T = T(W,\vec y_1) = \begin{bmatrix} \alpha_1 & \beta_1 \\
    \beta_1 & \alpha_2 & \ddots \\
    & \ddots & \ddots & \beta_{n-1} \\
    & & \beta_{n-1} & \alpha_n \end{bmatrix}
\end{align*}
and $T = Q W Q^*$ for some unitary matrix $Q$.  We use $T_k = T_k(W,\vec y_1), k = 1,2,\ldots,n$ to denote the upper-left $k\times k$ subblock of $T$.  Then, it is well-known that (see \cite{Greenbaum1989}, for example)
\begin{align}\label{eq:pks}
    p_k^\dagger(\lambda) = \frac{\det \left(T_k\left(W,\frac{\vec r_0}{\|\vec r_0\|} \right) - \lambda I \right)}{ \det T_k\left(W,\frac{\vec r_0}{\|\vec r_0\|} \right)},
\end{align}
where $\vec r_0 = \vec b - W \vec x_0$ as above.  Then, write $W = U \Lambda U^*, ~\Lambda = \mathrm{diag}(\lambda_1,\lambda_2,\ldots,\lambda_n)$ and for $\vec x_0 = \vec 0$, so that $\vec e_0 = \vec x$,
\begin{align*}
    \|\vec e_k\|_{W^{\ell}}^2 &= \| W^{\ell/2} p_k^\dagger(W) \vec e_0\|^2 = \| \Lambda^{\ell/2} p_k^\dagger(\Lambda) U^*\vec e_0\|^2= \| \Lambda^{\ell/2} p_k^\dagger(\Lambda) \Lambda^{-1} U^*\vec b\|^2\\
    &= \sum_{j=1}^n \lambda_j^{\ell-2} p_k^\dagger(\lambda_j)^2 \omega_j, \quad  \boldsymbol \omega = (\omega_j), \quad \omega_j = |(U^*\vec b)_j|^2, \quad 1 \leq j \leq n.
\end{align*}
Now, consider the special case $\vec b = \vec b_0 := [1,0,\ldots,0]^T$, so that $\vec r_0 = \vec b_0$. We further analyze the relation $Q^*TQ = W$ with $\vec y_1 = \vec b_0$.  The Lanczos algorithm gives the matrix representation of $W$ in the orthonormal basis found by applying the Gram--Schmidt procedure to the sequence
\begin{align*}
    \vec y_1, W \vec y_1, \ldots, W^{n-1} \vec y_1.
\end{align*}
So, the first vector is $\vec b_0$, and so the first column of $Q$ is $\vec b_0$.  The main consequence of this is that that first components of the eigenvectors\footnote{This is true modulo permutations and normalizations.} of $W$ are the same as those of $T$.  \\

\textbf{Basic assumption:} Henceforth, throughout the paper we will assume $\vec x_0 = \vec 0$.\\

Finally, we make a simple observation that the Householder bidiagonalization procedure \cite{TrefethenBau} applied to $X$ where $W = XX^*$ (or Householder tridiagonalization applied to $W$) leaves the eigenvalues of $W$ unchanged and also leaves the first components of the eigenvectors unchanged.  So, provided that Lanczos completes ($\beta_k \neq 0$ for $k = 1,\ldots,n-1$), the Householder bidiagonalization must produce $T(W,\vec b_0)$.  This is indeed true because a Jacobi matrix is uniquely defined by eigenvalues and first-components of eigenvectors (see, e.g.~\cite{Deift1985}).

\section{Main results}\label{sec:main}
The proof of our main theorem is given in Section~\ref{sec:proof}.  The convention used in this paper is that $\beta$ and $d$ are fixed constants. The symbols $C,c,C',c'$ with an assortment of subscripts will be used to denote constants and their (possible) dependencies.  We suppress any dependence of these constants on $\beta$ but include dependence on $d$, with a view to forthcoming work where we will allow $d$ to vary. 

\begin{theorem}\label{t:estimates}
Assume the conjugate gradient algorithm is applied to solve $W_{n,\beta,d} \vec x = \vec b$ where $\|\vec b\|_2 = 1$ is a (possibly) random vector, independent of $W_{n,\beta,d} = W$ and $0 < d <1$.  Let $\vec e_k = \vec x - \vec x_k$, $k = 0,1,2,\ldots$ be the associated error vectors.  
\begin{enumerate}
\item For any fixed $\ell \in \mathbb Z$ and $n > 1$ there exists a constant $C_{\ell,d,k}>0$ such that
\begin{align*}
    \mathbb E\left[ \left| \|\vec e_k\|_{W^{\ell}}^2 - \int \lambda^{\ell-2} \det(\mathbb T_{k,d} - \lambda I)^2 \mu_{\mathrm{MP},d}(\D \lambda) \right| \right] \leq C_{\ell,d,k} \frac{\log n}{\sqrt{n}}.
\end{align*}
\item Furthermore
\begin{align*}
    \mathbb P \left( \left| \|\vec e_k\|_{W^{\ell}}^2 -  \int \lambda^{\ell -2} \det (\mathbb T_{k,d} - \lambda I)^2 \mathbb E \mu_{n,\beta,d}(\D \lambda) \right| > t \right) \leq C'_{\ell,d,k} n \E^{- c_{\ell,d,k} n h(t)},
\end{align*}
for some constant $C'_{\ell,d,k} > 0$ and a non-decreasing function $h(t)$ that satisfies $h(t) > 0$ for $t > 0$.
\item Lastly, if $d =1$ and $\ell \geq 2$, (1) and (2) hold.
\end{enumerate}
\end{theorem}

\begin{theorem}\label{t:ek}
In the setting of Theorem~\ref{t:estimates}, for $0 < d < 1$ and $\ell \in \mathbb Z$, define\footnote{If $k = 0$, $\mathfrak e^2_{\ell,0,d} := \int \lambda^{\ell-2} \mu_{\mathrm{MP},d}(\D \lambda)$.}
\begin{align}\label{eq:klim}
    \mathfrak e^2_{\ell,k,d}:=\int \lambda^{\ell -2} \det(\mathbb T_{k,d} - \lambda I)^2 \mu_{\mathrm{MP},d}(\D \lambda).
\end{align}
Then as $k \to \infty$, $\mathfrak e^2_{\ell,k,d} \to 0$.  Furthermore,
\begin{align*}
    \mathfrak e^2_{1,k,d} = \frac{d^k}{1-d}.
\end{align*}
For $d = 1$, \eqref{eq:klim} holds if $ \ell \geq 2$, and for $0 < d \leq 1$
\begin{align*}
    \mathfrak e^2_{2,k,d} &= d^k,\\
    \mathfrak e^2_{3,k,d} &= d^k \begin{cases} 1 + d & k \geq 1,\\ 1 & k = 0. \end{cases}
\end{align*}
\end{theorem}

\begin{corollary}
In the setting of Theorem~\ref{t:estimates}, for $0 < d < 1$ and $\ell \in \mathbb Z$ and $n > 1$
\begin{align*}
    &\|\vec e_k\|_{W^{\ell}} \overset{\text{a.s.}}{\to} \mathfrak e_{\ell,k,d},\\
    &\mathbb E\left[ \left| \|\vec e_k\|_{W^{\ell}} - \mathfrak e_{\ell,k,d} \right|\right] \leq C_{\ell,d,k} \frac{\log n}{\sqrt{n}}.
\end{align*}
If $d =1$ these relations hold for $\ell \geq 2$.
\end{corollary}
\begin{proof}
The first claim follows from the Borel--Cantelli Lemma.  The second follows from the observation
\begin{align*}
    \|\vec e_k\|_{W^{\ell}}^2 - \mathfrak e_{\ell,k,d}^2 = \left( \|\vec e_k\|_{W^{\ell}} - \mathfrak e_{\ell,k,d}\right) \left(\|\vec e_k\|_{W^{\ell}} + \mathfrak e_{\ell,k,d} \right),
\end{align*}
which gives
\begin{align*}
    \left| \|\vec e_k\|_{W^{\ell}} - \mathfrak e_{\ell,k,d}\right| \leq \mathfrak e_{\ell,k,d}^{-1} \left| \|\vec e_k\|_{W^{\ell}}^2 - \mathfrak e_{\ell,k,d}^2 \right|.
\end{align*}
\end{proof}
The last of our main results is almost just a corollary of the above theorems and it concerns halting times
(i.e. runtimes or iteration counts, recall \eqref{eq:htimes}):
\begin{align*}
    t^{(1)}_{\epsilon} &= t^{(1)}_{\epsilon}(W_{n,\beta,d},\vec b) = \min\{ k: \|\vec e_k\|_{W_{n,\beta,d}} < \epsilon \},\\
    t^{(2)}_{\epsilon} &= t^{(2)}_{\epsilon}(W_{n,\beta,d},\vec b) = \min\{ k: \|\vec r_k\|_2 < \epsilon \}.
\end{align*}
Since  $\|\vec e_k\|_{W}$ converges almost surely to $\mathfrak e_{1,k,d}$ and $\|\vec r_k\|_{2} = \|\vec e_k\|_{W^2}$ converges almost surely to $\mathfrak e_{2,k,d}$ we produce the candidate limit halting times
\begin{align*}
    \tau^{(1)}_{\epsilon}(\beta,d) &= \left\lceil \frac{2\log \epsilon + \log (1-d)}{\log d} \right\rceil, \quad \epsilon^2 < (1-d)^{-1},\\
    \tau^{(2)}_{\epsilon}(\beta,d) &= \left\lceil \frac{2\log \epsilon}{\log d} \right\rceil, \quad \epsilon < 1.
\end{align*}

\begin{theorem}\label{t:halting}
In the setting of Theorem~\ref{t:estimates}, $0 < d <1$, for $\ell = 1,2$ suppose that $\epsilon \neq \mathfrak e_{\ell,k,d}$ for $k = 0,1,2,\ldots$, $\epsilon < \mathfrak e_{\ell,0,d}$, then\footnote{Note that $\epsilon < \mathfrak e_{\ell,0,d}$ is just a statement that $\epsilon^2 < (1-d)^{-1}$ for $\ell = 1$ and $\epsilon < 1$ for $\ell = 2$, and so $\tau_{\epsilon}^{(\ell)}(\beta,d) \geq 1$, $\ell = 1,2$.}
\begin{align*}
    \lim_{n \to \infty} \mathbb P \left( t^{(\ell)}_{\epsilon}(W_{n,\beta,d},\vec b) = \tau^{(\ell)}_{\epsilon}(\beta,d)  \right) = 1.
\end{align*}
If $\epsilon = \mathfrak e_{\ell,k,d}$, i.e. $k = \tau^{(\ell)}(\beta,d)$, for $k > 0$ then
\begin{align}\label{eq:2vals}
    \lim_{n \to \infty} \mathbb P \left( t^{(\ell)}_{\epsilon}(W_{n,\beta,d},\vec b) = \tau^{(\ell)}_{\epsilon}(\beta,d) ~~\text{ or }~~ t^{(\ell)}_{\epsilon}(W_{n,\beta,d},\vec b) = \tau^{(\ell)}_{\epsilon}(\beta,d) + 1 \right) = 1.
\end{align}
\end{theorem}
\begin{proof} 
Assume $\epsilon \neq \mathfrak e_{\ell,k,d}$ for $k = 1,2,\ldots$.  Then $\delta = \min_{k}|\epsilon^2 - \mathfrak e_{\ell,k,d}^2| > 0$ as $\mathfrak e_{\ell,k,d} \to 0$ as $k \to \infty$.  Note that if $\kappa = \tau^{(\ell)}_{\epsilon}(\beta,d)$ then 
\begin{align*}
    \mathfrak e_{\ell,\kappa-1,d}^2 > \epsilon^2  > \mathfrak e_{\ell,\kappa,d}^2.
\end{align*}
Then as $\delta \leq |\epsilon^2 - \mathfrak e_{\ell,k-1,d}^2| = \mathfrak e_{\ell,k-1,d}^2 - \epsilon^2$
\begin{align*}
    \mathbb P \left( t^{(\ell)}_{\epsilon} \neq \tau^{(\ell)}_{\epsilon}(\beta,d)  \right) \leq \mathbb P \left( t^{(\ell)}_{\epsilon} < \tau^{(\ell)}_{\epsilon}(\beta,d)  \right) + \mathbb P \left( t^{(\ell)}_{\epsilon} > \tau^{(\ell)}_{\epsilon}(\beta,d)  \right)
\end{align*}
We estimate these two terms individually. First, {let $\mathcal M_k$ be the event where $\|\vec e_{j}\|_{W^\ell}$ is weakly decreasing for $j = 0,1,2,\ldots,k-1$.  Then
\begin{align*}
    \mathbb P \left( t^{(\ell)}_{\epsilon} < \tau^{(\ell)}_{\epsilon}(\beta,d)  \right) &= \mathbb P \left( \|\vec e_{k-1}\|_{W^\ell}^2 < \epsilon^2, k = \tau^{(\ell)}_{\epsilon}(\beta,d), \mathcal M_k \right)  + \mathbb P \left( t^{(\ell)}_{\epsilon} < \tau^{(\ell)}_{\epsilon}(\beta,d), \mathcal M_k^c  \right)  \\
    &\leq \mathbb P \left( \|\vec e_{k-1}\|_{W^\ell}^2 < \mathfrak e_{\ell,k-1,d}^2 - \delta, k = \tau^{(\ell)}_{\epsilon}(\beta,d)  \right)+ \mathbb P (\mathcal M_k^c).
\end{align*}}
For sufficiently large $n$, by Lemma~\ref{l:powers}, 
\begin{align*}
    \left| \mathfrak e_{\ell,k-1,d}^2 - \int \lambda^{\ell-2} \det (\mathbb T_{k-1,d} - \lambda I)^2 \mathbb E \mu_{n,\beta,d}(\D \lambda) \right| \leq \frac{\delta}{2}
\end{align*}
and for such a value of $n$
\begin{align*}
    &\mathbb P \left( \|\vec e_{k-1}\|_{W^\ell}^2 < \mathfrak e_{\ell,k-1,d}^2 - \delta, k = \tau^{(\ell)}_{\epsilon}(\beta,d)  \right) \\&\leq \mathbb P \left( \|\vec e_{k-1}\|_{W^\ell}^2 < \int \lambda^{\ell-2} \det (\mathbb T_{k-1,d} - \lambda I)^2 \mathbb E \mu_{n,\beta,d}(\D \lambda) - \frac{\delta}{2}, k = \tau^{(\ell)}_{\epsilon}(\beta,d)  \right) \to 0
\end{align*}
by Theorem~\ref{t:estimates}(2). {It remains to show that $\mathbb P(\mathcal M_k^c) \to 0$.  For $\ell = 1$ this is immediate because $P(\mathcal M_k^c)  = 0$.  For $\ell = 2$, consider the event $\mathcal S_k = \{ d^{j+1/2} <  \|\vec e_j\|^2_{W^2} < d^{j-1/2},~~ j = 1,2,\ldots,k-1  \}$. Then $\mathbb P(\mathcal M_k^c) \leq \mathbb P(\mathcal S_k^c)$ and
\begin{align*}
    \mathbb P(\mathcal S_k^c) \leq \sum_{j=1}^{\tau^{(2)}_{\epsilon}(\beta,d)-1} \left[\mathbb P ( \|\vec e_j\|^2_{W^2} \geq d^{j-1/2} ) + \mathbb P ( \|\vec e_j\|^2_{W^2} \leq d^{j+1/2} ) \right].
\end{align*}
And this tends to zero by Theorem~\ref{t:estimates}(2).}
The estimate for $\mathbb P \left( t^{(\ell)}_{\epsilon} > \tau^{(\ell)}_{\epsilon}(\beta,d) \right)$ is analogous, but now we do not need monotonicity for $\|\vec e_j\|_{W^\ell}$ and 
\begin{align*}
    \mathbb P \left( t_\epsilon^{(\ell)} (W_{n,\beta,d}, \vec b) > \tau_{\epsilon}^{(\ell)}(\beta,d) \right) &= \mathbb P\left( \|\vec e_k\|^2_{W^{\ell}} \geq \epsilon^2, k = \tau_{\epsilon}^{(\ell)}(\beta,d) \right)\\
    & \leq \mathbb P \left( \| \vec e_k\|^2_{W^{\ell}} \geq \mathfrak e_{\ell,k,d}^2 + \delta, k = \tau_{\epsilon}^{(\ell)}(\beta,d)\right)
\end{align*}
as $\delta \leq |\epsilon^2 - \mathfrak e^2_{\ell,k,d}| = \epsilon^2 - \mathfrak e_{\ell,k,d}^2$.  When $\epsilon = \mathfrak e_{\ell,k,d}$, we use similar arguments to show that
\begin{align*}
\lim_{n \to \infty} \mathbb P \left(t_\epsilon^{(\ell)}(W_{n,\beta,d}, \vec b) < \tau_{\epsilon}^{(\ell)}(\beta,d)  \right) = \lim_{n \to \infty} \mathbb P \left(t_\epsilon^{(\ell)}(W_{n,\beta,d}, b) > \tau_{\epsilon}^{(\ell)}(\beta,d) + 1  \right) = 0.
\end{align*}
As $t_{\epsilon}^{(\ell)}(W,\vec b)$ and $\tau_{\epsilon}^{(\ell)}(\beta,d)$ are integers, \eqref{eq:2vals} follows.
\end{proof}
\begin{remark}
We conjecture that if $\epsilon = \mathfrak e_{\ell,k,d}$, i.e $k = \tau^{(\ell)}(\beta,d)$, for $k > 0$ then
\begin{align*}
    \lim_{n \to \infty} \mathbb P \left( t^{(\ell)}_{\epsilon}(W_{n,\beta,d},\vec b) = \tau^{(\ell)}_{\epsilon}(\beta,d) \right) = \frac 1 2 =  \lim_{n \to \infty} \mathbb P \left( t^{(\ell)}_{\epsilon}(W_{n,\beta,d},\vec b) = \tau^{(\ell)}_{\epsilon}(\beta,d) + 1 \right).
\end{align*}
Indeed Figure~\ref{f:clt} indicates this is true because
\begin{align*}
    \|\vec r_k\|_{2}^2 - d^{k}
\end{align*}
appears to be asymptotically normal with a variance that decays like $1/n$.  We note that this is related to, but not a consequence of, the central limit theorem for linear spectral statistics (CLT for LSS).  For the CLT for LSS the variance decays as $1/n^2$. In the case at hand, the fluctuations that occur in the random weights (see $\omega_j$ in \eqref{eq:nu}) and the fluctuations that occur in the random polynomial $p_k^\dagger$ contribute to the variance on the order of $1/n$.  This conjecture will be resolved in a forthcoming publication.
\end{remark}

\begin{figure}[tbp]
\begin{overpic}[width=.8\linewidth]{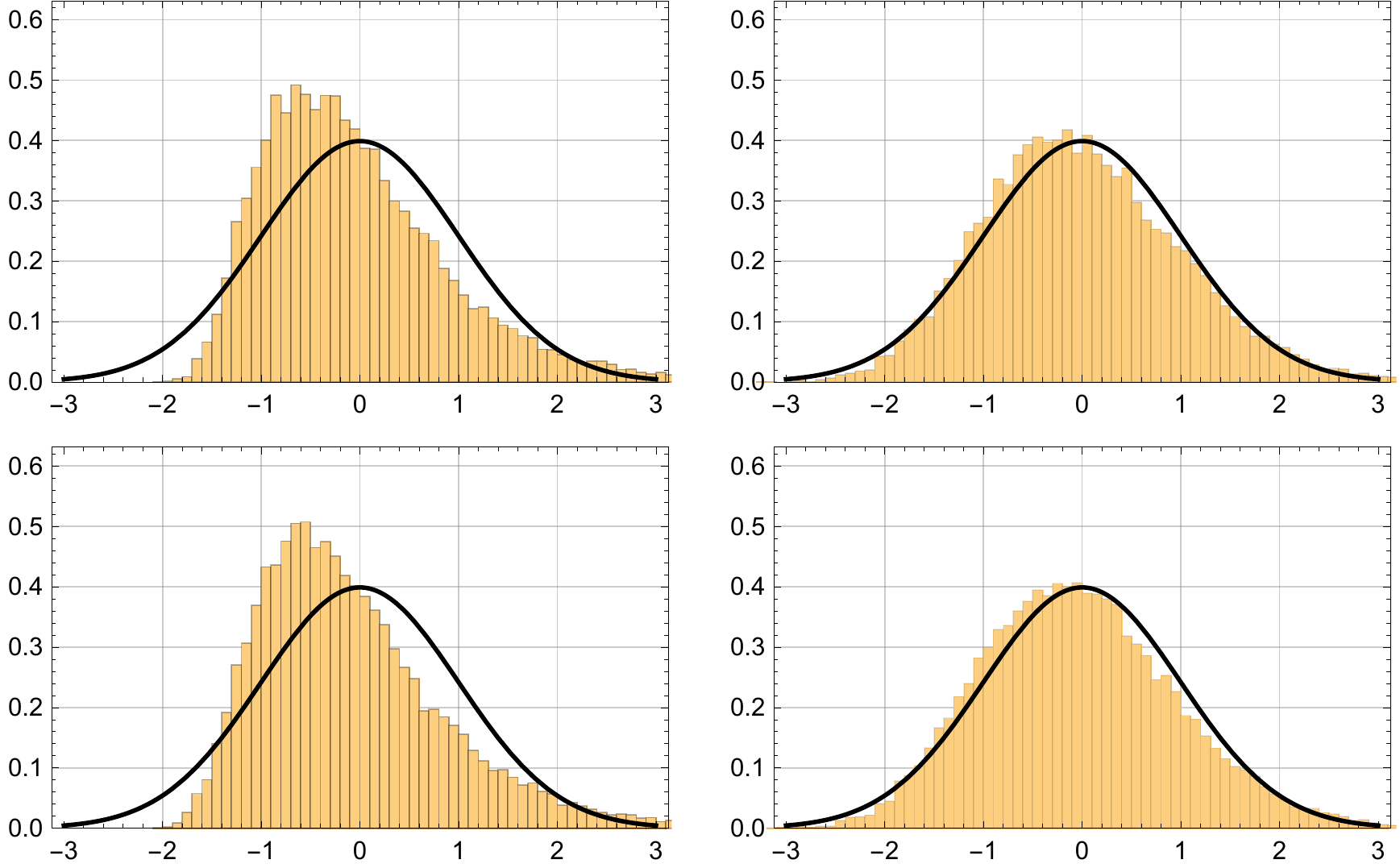}
\put(38,25){BE}
\put(34,21){$n = 200$}
\put(89.5,25){BE}
\put(85,21){$n = 2000$}

\put(38,57){$\beta=1$}
\put(34,53){$n = 200$}
\put(89.5,57){$\beta=1$}
\put(85,53){$n = 2000$}

\put(33,-4){Normalized fluctuations}

\put(-4,20){\rotatebox{90}{Relative frequency}}
\end{overpic}
\vspace{.1in}
\caption{\label{f:clt} A demonstration that $g_k : = \|\vec r_k\|_{2}^2 - d^{k}$ is appears asymptotically normal as $n$ increases.  The top row demonstrates this for case $\beta = 1$ and the bottom row demonstrates this for the case where $X_{11} = \pm 1$ with equal probability ($X$ still has iid entries), \emph{i.e.} the Bernoulli Ensemble. Specifically, we plot histograms for $g_k/(\langle g_k^2 \rangle)^{1/2}$ against a standard Gaussian density (black). }
\end{figure}
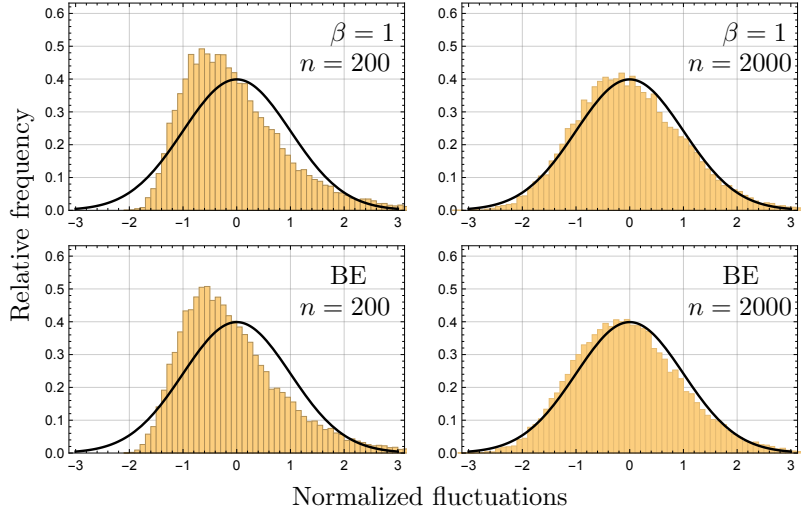
\section{Technical results from random matrix theory}\label{sec:tech}

\begin{lemma}\label{l:chibound}
Let $\chi_k$ be a chi distributed random variable with $k \geq 1$ degrees of freedom.  Then for any fixed  integers $p$ and $q > 0$ there exists $C_{q,p} > 0$ such that
\begin{align*}
    \mathbb E \left| \left( \frac{\chi_k}{\sqrt{k} } \right)^{p} - 1 \right|^{2q} \leq C_{q,p} k^{-q}, \quad k > -2pq.
\end{align*}
Furthermore, for $t \geq 0$
\begin{align*}
    \mathbb P\left(  \frac{\chi_k}{\sqrt{k}} - 1  \geq t \right) \leq  \frac{1 + O(k^{-1})}{\sqrt{\pi}} \int_{t\sqrt{k}}^\infty \E^{-x^2/2} \D x.
\end{align*}
For $ t \leq 0$,
\begin{align*}
    \mathbb P\left(  \frac{\chi_k}{\sqrt{k}} - 1  \leq t \right) \leq  \frac{\E^{1/2} + O(k^{-1})}{\sqrt{\pi}} \int_{-\infty}^{t\sqrt{k-1}}  \E^{-x^2/2} \D x.
\end{align*}
\end{lemma}
\begin{proof}
Because $\chi_k$ has a density given by
\begin{align*}
    \frac{x^{k-1}\E^{-x^2/2} }{2^{k/2-1} \Gamma(k/2)}
\end{align*}
we are led to analyze
\begin{align*}
\frac{1}{{2^{k/2-1} \Gamma(k/2)}} \int_0^\infty \left| \left( \frac{x}{\sqrt{k} } \right)^{p} - 1 \right|^{2q} x^{k-1}\E^{-x^2/2} \D x.
\end{align*}
The result then follows by the change of variables $x = \sqrt{k} y$ and applying the method of steepest descent (Laplace's method, see e.g., \cite[Lemma 6.2.3]{FokasComplexVariables}) for integrals along with Stirling's approximation.  Indeed, suppose $f :\mathbb (0,\infty) \to \mathbb R $ is smooth and satisfies the bound $|f(y)| \leq C (y^{-K} + y^L)$ for $K,L, C > 0$.  Then for $k > K$ we must estimate
\begin{align*}
    \int_0^\infty f(y) y^{k - 1} \E^{-k y^2/2} \D y &= \E^{-k/2} \int_0^\infty f(y) y^{- 1} \E^{-k \ell(y)} \D y, 
    \\ \ell(y) &= y^2/2 - \log y - 1/2.
\end{align*}
Note that $\ell(y)$ has a global minimum of zero at $y = 1$ on $(0,\infty)$.  For $0 < \epsilon < 1$, write
\begin{align*}
    \int_0^{1-\epsilon} f(y) y^{- 1} \E^{-k \ell(y)} \D y & = \int_0^{1-\epsilon} f(y) y^{- 1} \E^{-(k - K - 1) \ell(y)} \E^{-(K + 1) \ell(y)} \D y\\
    & \leq \E^{-(k - K - 1) \ell(1-\epsilon)} \int_0^{\infty} f(y) y^{- 1}  \E^{-(K + 1) \ell(y)} \D y,
\end{align*}
which decays exponentially to zero as $\ell(1-\epsilon) > 0$.  The same calculation on $[1 + \epsilon, \infty)$ gives
\begin{align*}
    \int_{1+\epsilon}^\infty f(y) y^{- 1} \E^{-k \ell(y)} \D y  \leq \E^{-(k - K - 1) \ell(1+\epsilon)} \int_0^{\infty} f(y) y^{- 1}  \E^{-(K + 1) \ell(y)} \D y,
\end{align*}
which, again, decays exponentially because $\ell(1 + \epsilon) > 0$.  Laplace's method gives
\begin{align*}
    \int_0^\infty f(y) y^{k - 1} \E^{-k y^2/2} \D y = \E^{-k/2}\left( f(1)  \sqrt{\frac{\pi}{k }} + O(k^{-3/2})\right).
\end{align*}
As remarked, Stirling's approximation then gives the first inequality in the lemma.
We note that
\begin{align}\label{eq:ineqy}
    \ell(y+1) = \frac{(y+1)^2}{2} - \log (y+1)  -1/2 \geq \frac{y^2}{2}, \quad y \geq -1.
\end{align}
For the second inequality then uses that \eqref{eq:ineqy}
\begin{align*}
    \mathbb P\left(  \frac{\chi_k}{\sqrt{k}} - 1  \geq t \right) &= \frac{k^{k/2}}{2^{k/2-1}\Gamma(k/2)} \int_{t+1}^\infty x^{k-1} \E^{-k x^2/2} \D x\\
    &= \frac{k^{k/2}}{2^{k/2-1}\Gamma(k/2)} \int_{t}^\infty (y+1)^{k-1} \E^{-k (y+1)^2/2} \D x\\
    & = \frac{k^{k/2}\E^{-k/2}}{2^{k/2-1}\Gamma(k/2)} \int_{t}^\infty (y+1)^{-1} \E^{-k \ell(y+1)} \D y\\
    & \leq \frac{k^{k/2}\E^{-k/2}}{2^{k/2-1}\Gamma(k/2)} \int_{t}^\infty \E^{-k y^2/2} \D y.
\end{align*}
The last follows from \eqref{eq:ineqy} and a similar calculation
\begin{align*}
    \mathbb P\left(  \frac{\chi_k}{\sqrt{k}} - 1  \leq t \right) & = \frac{k^{k/2}\E^{-k/2}}{2^{k/2-1}\Gamma(k/2)} \int_{-1}^t \E^{-(y+1)^2/2 + 1/2} \E^{-(k-1) \ell(y+1)} \D y\\
    & \leq \E^{1/2}\frac{k^{k/2}\E^{-k/2}}{2^{k/2-1}\Gamma(k/2)} \int_{-1}^t \E^{-(k-1) \ell(y+1)} \D y\\
    & \leq \E^{1/2}\frac{k^{k/2}\E^{-k/2}}{2^{k/2-1}\Gamma(k/2)} \int_{-\infty}^t \E^{-(k-1) y^2/2} \D y.
\end{align*}
\end{proof}

\begin{definition}
  A mean-zero random variable $X$ is called sub-exponential with parameters $\nu,\alpha > 0$ if $\mathbb E \left[\E^{\lambda x}\right]\leq \E^{\lambda^2 \nu^2/2}$ for $|\lambda| < \frac{1}{\alpha}$.
\end{definition}
It then follows that centered chi variables are sub-exponential. If $X$ is sub-exponential, then clearly $\mathbb E \left[ \E^{|X|/t} \right] < 2$ for some $ t > 0$.  

A good reference for the next classical result is \cite[Section 2.8]{Vershynin2018}.
\begin{theorem}[Bernstein's inequality for sub-exponential random variables]
Let $(X_i)_{i \geq 1}$ be a sequence of independent mean zero random variables and define
\begin{align*}
    \|X_i\|_{\psi_1} := \inf\{t > 0: \mathbb E \exp(|X|/t) \leq 2\}.
\end{align*}
Then for $t \geq 0$ and any real numbers $a_1,\ldots,a_n$
\begin{align*}
    \mathbb P \left(  \left|\sum_{j=1}^n a_i X_i \right| \geq t \right) \leq 2 \exp \left( - c \min \left\{ \frac{t^2}{K^2\|\vec a\|_2^2}, \frac{t}{K \|\vec a\|_\infty} \right\} \right),
\end{align*}
and $K = \max_{1 \leq i \neq n} \|X_i\|_{\psi_1}$.  Here $c > 0$ is some absolute constant independent of $\{X_i\}, \{a_j\}$.
\end{theorem}

The estimate
\begin{align}\label{eq:chi-est}
    \mathbb P \left( \left| \frac{\chi_{\beta n}^2}{\beta n} - 1 \right| \geq t \right) \leq 2 \exp \left ( - c n\min \left\{ \frac{ \beta^2 t^2 }{K^2}, \frac{\beta t}{K} \right\} \right),
\end{align}
can be found by estimating the density for a $\chi_{\beta n}^2$ random variable or by applying Bernstein's inequality (Recall that a sum of $n$ independent chi-square variables $\chi_{\sigma_i}^2$, $i = 1,\ldots,n$ is again a chi-square variable $\chi_\sigma^2$ with $\sigma = \sum_i \sigma_i$). We will use three elementary facts that are encapsulated in the following lemma.
\begin{lemma}\label{l:3facts}
Let $Z_1,Z_2,Y$ be random variables and assume $\mathbb P( Y = 0 ) = 0$.  The following inequalities hold
\begin{enumerate}
    \item $\displaystyle \mathbb P \left( \frac{|Z_1|}{|Y|} \geq t \right) \leq \mathbb P \left( |Z_1| \left[ \frac{1}{|Y|} -\frac{1}{\mu} \right]_+ + \frac{|Z_1|}{\mu} \geq t \right)$
    where $[\cdot]_+$ denotes the positive part and $\mu > 0$,\\
    \item $\displaystyle \mathbb P(|Z_1| + |Z_2| \geq t) \leq \mathbb P(|Z_1| \geq t/p) + \mathbb P(|Z_2| \geq t/q)$, ~~ $1/p + 1/q = 1$ and\\
    \item $\displaystyle \mathbb P(|Z_1||Z_2| \geq t) \leq \mathbb P(|Z_1| \geq t^{1/2}) + \mathbb P(|Z_2| \geq t^{1/2})$.
\end{enumerate}
\end{lemma}

\begin{lemma}\label{l:KSdist}
Suppose $-\infty < \lambda_{1,n} < \lambda_{2,n} < \cdots \lambda_{n,n} < \infty$.  Let $(\chi_{\beta}^{(j)})_{j \geq 1}$ be independent chi-distributed random variables with $\beta$ degrees of freedom.  Define weights
\begin{align*}
    \omega_j = \frac{(\chi_\beta^{(j)})^2}{\sum_k (\chi_\beta^{(k)})^2}.
\end{align*}
Then the Kolmogorov--Smirnov distance 
\begin{align*}
d_{\mathrm{KS}}(\mu,\nu):=\sup_{x \in \mathbb R} |\mu((-\infty,x]) - \nu((-\infty,x])|
\end{align*}
of
\begin{align*}
    \mu_n &= \sum_{j} \delta_{\lambda_{j, n}} \omega_j, \quad \nu_n = \frac{1}{n} \sum_j \delta_{\lambda_{j, n}},
\end{align*}
satisfies
\begin{align*}
   \mathbb E \left[ d_{\mathrm{KS}}(\mu_n,\nu_n) \right] &\leq C\frac{\log n}{\sqrt{n}},
\end{align*}
and the tail estimate
\begin{align}\label{eq:KStail}
    \mathbb P( d_{\mathrm{KS}}(\mu_n,\nu_n)  \geq  t) \leq  C_1 n e^{-c_1 n \beta t^2} +  C_2 n e^{-c_2 n \beta t}
\end{align}
for absolute constants $C,C_1,C_2,c_1,c_2>0$.
\end{lemma}
Note that $ d_{\mathrm{KS}}(\mu_n,\nu_n) \leq 1$.
\begin{proof}
First, it follows that
\begin{align*}
    d_{\mathrm{KS}}(\mu_n,\nu_n) &= \max_{1 \leq k \leq n} \left| \sum_{j=1}^k \left( \omega_j - \frac{1}{n} \right) \right|\\
    & = \max_{1 \leq k \leq n} \frac{1}{\sum_\ell (\chi_\beta^{(\ell)})^2}\left| \sum_{j=1}^k \left( (\chi_\beta^{(j)})^2 - \frac{1}{n}\sum_i (\chi_\beta^{(i)})^2 \right) \right|\\
    & = \max_{1 \leq k \leq n} \frac{1}{\sum_\ell (\chi_\beta^{(\ell)})^2}\left| \sum_{j=1}^k \left( \frac{n-k}{n}   \right)(\chi_\beta^{(j)})^2 - \sum_{j=k+1}^n \frac{k}{n} (\chi_\beta^{(j)})^2  \right|.
\end{align*}
So, we are led to analyze the sums
\begin{align*}
    S_k &= \sum_{j=1}^k \left( \frac{n-k}{n}   \right)(\chi_\beta^{(j)})^2 - \sum_{j=k+1}^n \frac{k}{n} (\chi_\beta^{(j)})^2,\\
    S &= \sum_{j=1}^n (\chi_{\beta}^{(j)})^2.
\end{align*}
As $S_k$ has expected value zero,  Bernstein's inequality gives for $ t \geq 0$
\begin{align*}
    \mathbb P ( |S_k| \geq t ) &\leq 2 \exp\left( -c \min\left\{ \frac{4t^2}{K^2 n}, \frac{t}{K} \right\} \right)=:F(t).
\end{align*}
for absolute constants $c,K> 0$.  From the moment generating function for a chi-square distribution, we have, as $S$ is a chi-square random variable with $n \beta$ degrees of freedom,
\begin{align*}
    \mathbb P( S \leq t) \leq \min_{s > 0} \E^{st} (1 + 2s)^{-\frac{n\beta}{2}}.
\end{align*}
This minimum occurs at $s = \frac{n \beta}{2 t} - \frac 1 2$, giving
\begin{align*}
    \mathbb P( S \leq n \beta t) \leq \E^{\frac{n \beta}{2}(1 - t)} t^{\frac{n \beta}{2}} = \left( t \E^{1-t} \right)^{\frac{n\beta}{2}}.
\end{align*}
Then we write for $0 \leq s \leq 1$
\begin{align*}
     \mathbb P( S \leq n \beta (1-s)) \leq \E^{\frac{n \beta}{2}(s)} (1-s)^{\frac{n \beta}{2}},
\end{align*}
so that 
\begin{align*}
    \mathbb P\left( \frac{1}{S} - \frac{1}{n \beta} \geq \frac{1}{n \beta} \left[ \frac{1}{1-s} - 1 \right]  \right) \leq \E^{\frac{n \beta}{2}s} (1-s)^{\frac{n \beta}{2}}.
\end{align*}
Now set $t = \frac{1}{1-s} - 1 = \frac{s}{1-s}$, $s = \frac{t}{t+1}$ to find
\begin{align*}
    \mathbb P\left( \frac{1}{S} - \frac{1}{n \beta} \geq \frac{t}{n \beta} \right) \leq \left( \E^{\frac{t}{t+1}}\left[ 1 - \frac{t}{t+1} \right]\right)^{\frac{n\beta}{2}}.
\end{align*}
Then it is easy to see that for $0 \leq t \leq 1$
\begin{align*}
    \E^{\frac{t}{t+1}}\left[ 1 - \frac{t}{t+1} \right] \leq \E^{-t^2/8},
\end{align*}
giving the estimate
\begin{align*}
    \mathbb P\left(  \left[\frac{n \beta}{S} - 1\right]_+ \geq t \right) \leq \E^{-n \beta t^2/16} =: G(n\beta; t).
\end{align*}
We then write $\tilde S_k = S_k/(n\beta)$ and $\tilde S = S/(n \beta)$ so that
\begin{align*}
    \frac{|S_k|}{|S|} = \frac{|\tilde S_k|}{|\tilde S|}
\end{align*}
and then we apply each property of Lemma~\ref{l:3facts}, in order, to obtain
\begin{align*}
    \mathbb P \left( \frac{|S_k|}{|S|} \geq t \right) \leq G\left(n \beta; \frac{t^{1/2}}{\sqrt{p}} \right) + F\left(n \beta \frac{t^{1/2}}{\sqrt{p}} \right) + F\left(n \beta \frac{t}{q} \right),
\end{align*}
for $1/p + 1/q = 1$.  The tail estimate \eqref{eq:KStail} follows by a union bound. We examine $F$ more closely, and get a crude bound
\begin{align*}
    F\left(n \beta \frac{t^{1/2}}{\sqrt{p}} \right) &= 2 \exp \left( - \frac{c}{Kp^{1/2}} n \beta \min \left\{ \frac{4 \beta t}{Kp^{1/2}}, t^{1/2} \right\} \right) \\
    &\leq 2 \exp \left( - \frac{ct}{Kp^{1/2}} n \beta \min \left\{ \frac{4 \beta }{Kp^{1/2}}, 1 \right\} \right), \quad 0 \leq t \leq 1.
\end{align*}
While we do not specifically need the value, it follows that for a $\chi$-squared random variable $\chi_\beta^2$ with $\beta$ degrees of freedom $\|\chi_\beta^2\|_{\psi_1} = \frac{2}{1 - (1/2)^{2/\beta}}$ gives $K = \|\chi_\beta^2 - \beta\|_{\psi_1} < \infty$.  In summary, we obtain
\begin{align*}
    F\left(n \beta \frac{t^{1/2}}{\sqrt{p}} \right) \leq 2 \exp( - C n t), \quad 0 \leq t \leq 1.
\end{align*}
Then, we can estimate moments
\begin{align*}
    \mathbb E\left[ \left |\frac{S_k}{S}\right|^m \right] &= m \int_0^1 t^{m-1} \mathbb P \left( \frac{|S_k|}{|S|} \geq t \right) \D t,\\
    & \leq \frac{c_1 m}{n^m} \Gamma(m) + \frac{c_2 m}{n^m} \Gamma(m) + \frac{c_3 m}{n^{m/2}} \Gamma\left( \frac m 2\right),
\end{align*}
where $\Gamma$ denotes the Gamma function.  As $(\alpha_1 + \cdots + \alpha_k)^{1/m} \leq \alpha_1^{1/m} + \cdots + \alpha_k^{1/m}$ for all $\alpha_i \geq 0$, it follows that $\left( \mathbb E \left[ \left|\frac{S_k}{S}\right|^m \right] \right)^{1/m} \leq C' \frac{m}{\sqrt{n}}$ for some $C' > 0$.

Thus $\|n^{1/2}\left |\frac{S_k}{S}\right|\|_{\psi_1} \leq C''$ for some absolute constant $C''$.  By Jensen's inequality
\begin{align*}
    \mathbb E \left[ \max_k |X_k| \right] \leq s \log \mathbb E \left[ \exp  \max_k |X_k|/s \right] \leq s \log \sum_k \mathbb E \left[ \exp  |X_k|/s \right],
\end{align*}
and choosing $s = \max_k \|X_k\|_{\psi_1}$ we obtain
\begin{align*}
     \mathbb E \left[ \max_k |X_k| \right] \leq \log 2 n( \max_k \|X_k\|_{\psi_1}).
\end{align*}
Thus $\mathbb E [ d_{\mathrm{KS}}(\mu_n,\nu_n) ] \leq C \frac{\log n}{n^{1/2}}$,  for some new constant $C$.  Hence $d_{KS}(\mu_n,\nu_n)$ converges to zero in $L^1$, in probability and almost surely\footnote{Because $d_{\mathrm{KS}}(\mu,\nu)$ is always less than or equal to unity, almost sure convergence gives $L^1$ convergence, but we have obtained a rate.}.
\end{proof}

\begin{theorem}[Global eigenvalue bounds, see, e.g. \cite{Davidson2001}]\label{t:geb}
For the eigenvalues $\lambda_n \leq \cdots \leq \lambda_1$ of a $\beta$-Wishart distribution
\begin{align*}
    \mathbb P \left( 1 - \sqrt{\frac{n}{m}} - t \leq \lambda_n^{1/2} \leq \lambda_1^{1/2} \leq 1 + \sqrt{\frac{n}{m}} + t\right) \geq 1- 2 \E^{-c n t^2},
\end{align*}
for an absolute constant $c$.
\end{theorem}
This immediately implies that for any interval $(a,b)$ such that $[(1-\sqrt{d})^2,(1+\sqrt{d})^2] \subset (a,b)$ there exists a constant $ \gamma = \gamma(a,b) > 0$ such that 
\begin{align}\label{eq:deviation}
    \mathbb P ( \lambda_n < a \text{ or } \lambda_1 > b ) \leq 2 \E^{-n \gamma}.
\end{align}
And it also implies the bound on the distribution function for $\lambda_n$.  Define $d_n = \frac{n}{m} = d + o(1)$ so
\begin{align*}
    \mathbb P(\lambda_1 \geq t) \leq \begin{cases} 1 & t \leq (1 + \sqrt{d_n})^2\\
    2 \E^{-c n (t^{1/2}-1-\sqrt{d_n})^2} & t > (1 + \sqrt{d_n})^2 \end{cases} \leq \begin{cases} 1 & s_+ \leq 0,\\
    2 \E^{-n c \min\{s_+^2/32,s_+/4\}} & s_+ > 0,\end{cases}\\
    \mathbb P(\lambda_n \leq t) \leq \begin{cases} 1 & t \geq (1 - \sqrt{d_n})^2\\
    2 \E^{-c n (t^{1/2}-1+\sqrt{d_n})^2} & t < (1 - \sqrt{d_n})^2 \end{cases} \leq \begin{cases} 1 & s_- \geq 0,\\
    2 \E^{-n c s_-^2/32} & s_- < 0,\end{cases}
\end{align*}
where $s_\pm = t - (1 \pm  \sqrt{d_n})^2$. And the important conclusion from this is that
\begin{align}\label{eq:largest}
    \mathbb E[\lambda_1^k] \leq C_k < \infty, \quad k = 0,1,2,\ldots,
\end{align}
where $C_k$ is independent of $n$.  We give an analogue of \eqref{eq:largest} for $k < 0$ with $\lambda_1$ replaced with $\lambda_n$ in \eqref{eq:invbound}.

{\begin{lemma}\label{l:marginal_bound}
The marginal density $R(\mu)$ of the smallest eigenvalue of $\beta m W_{n,\beta,d}$ satisfies
\begin{align*}
    R(\mu) \leq n 2^{-p} \frac{\Gamma\left( 1 + \frac \beta 2 \right)}{\Gamma \left( 1 + \frac{n\beta}{2}\right)}\frac{\Gamma\left(\frac{\beta}{2}(m+1)  \right)}{\Gamma\left(p \right)\Gamma\left( p + \frac{\beta}{2}\right)} \mu^{p-1} \E^{-\mu/2}, \quad p = \frac{\beta}{2}(m-n + 1).
\end{align*}
Moreover, if $m = \lfloor n/d \rfloor$, $0 < d < 1$
\begin{align*}
R(\mu) \leq D_{n,d}^\beta \frac{\mu^{p-1} \E^{-\mu/2}}{2^p \Gamma(p)},
\end{align*}
where
\begin{align}\label{eq:growth}
    D_{n,d}^\beta = \frac{\E + o(1)}{\sqrt{2 \pi (1-d)}}\frac{\Gamma\left( 1 + \frac \beta 2 \right)}{\frac{\beta}{2} p^{\beta/2}} \sqrt{\frac{n \beta}{2}} \E^{ \frac{n \beta}{2}\left[ \left( \frac 1 d - 1\right) \log (1-d)^{-1} + \log d^{-1} \right]}.
\end{align}
\end{lemma}
\begin{proof}
We follow \cite{Edelman2005a}.  Define the multivariate Gamma function
\begin{align*}
    \Gamma^\beta_m(z) = \pi^{\beta m(m-1)/4} \prod_{\ell=1}^m \Gamma \left( z + \frac{\beta}{2} (\ell -1 ) \right).
\end{align*}
Then the joint density of the eigenvalues $ \mu_n \leq \cdots \leq \mu_1$ of $\beta m W_{n,\beta,d}$ is
\begin{align*}
   P(\mu_1, \ldots, \mu_n) = n! c_{n,m}^\beta \prod_{j < \ell} |\mu_\ell - \mu_j|^\beta \prod_\ell \mu_\ell^{p -1} \E^{- \sum_{\ell} \mu_\ell /2},
\end{align*}
where
\begin{align*}
    p = \frac{\beta}{2} (m - n + 1), \quad c^\beta_{n,m} = \frac{\pi^{\beta n (n-1)/2}}{2^{\beta m n/2}} \frac{ \Gamma\left( 1 + \frac{\beta}{2} \right)^n }{\Gamma_n^\beta\left( 1 + \frac \beta 2\right) \Gamma_n^\beta(p)}.
\end{align*}
Then
\begin{align*}
    R(\mu_n) &:= \int_{\mu_1 \geq \cdots \geq \mu_{n-1} \geq \mu_n \geq 0} P(\mu_1, \ldots, \mu_n) \D \mu_1 \cdots \D \mu_{n-1} \\
    &\leq \int_{\mu_1 \geq \cdots \geq \mu_{n-1} \geq 0} n! c_{n,m}^\beta \mu_n^{p -1} \E^{- \mu_n/2} \prod_{j < \ell < n} |\mu_\ell - \mu_j|^\beta \\
    & \times \prod_{\ell > 1} \mu_\ell^{p -1 + \beta} \E^{- \sum_{\ell < n} \mu_\ell /2}\D \mu_2 \cdots \D \mu_n\\
    & = \mu_n^{p -1} \E^{- \mu_n/2} \frac{n! c_{n,m}^\beta }{(n-1)! c_{n-1,m+1}^\beta}.
\end{align*}
Then we have
\begin{align*}
    Z_{n,m}^\beta&:=\frac{n! c_{n,m}^\beta }{(n-1)! c_{n-1,m+1}^\beta}\\
    &= n \frac{\pi^{\beta n (n-1)/2}}{2^{\beta m n/2}} \frac{ \Gamma\left( 1 + \frac{\beta}{2} \right)^n }{\Gamma_n^\beta\left( 1 + \frac \beta 2\right) \Gamma_n^\beta(p)}   \frac{2^{\beta (m+1)(n-1)/2}}{\pi^{\beta (n-1) (n-2)/2}} \frac{\Gamma_{n-1}^\beta\left( 1 + \frac \beta 2\right) \Gamma_{n-1}^\beta(p+\beta)}{ \Gamma\left( 1 + \frac{\beta}{2} \right)^{n-1} }.
\end{align*}
Define the modified multivariate Gamma function
\begin{align*}
    \tilde \Gamma^\beta_m(z) =  \prod_{\ell=1}^m \Gamma \left( z + \frac{\beta}{2} (\ell -1 ) \right),
\end{align*}
and we have
\begin{align*}
    Z_{n,m}^\beta &= n \frac{2^{\beta (m+1)(n-1)/2}}{2^{\beta m n/2}} \frac{ \Gamma\left( 1 + \frac{\beta}{2} \right)^n }{\tilde \Gamma_n^\beta\left( 1 + \frac \beta 2\right) \tilde \Gamma_n^\beta(p)} \frac{\tilde \Gamma_{n-1}^\beta\left( 1 + \frac \beta 2\right) \tilde \Gamma_{n-1}^\beta(p+\beta)}{ \Gamma\left( 1 + \frac{\beta}{2} \right)^{n-1} }\\
    & = n \frac{2^{\beta (m+1)(n-1)/2}}{2^{\beta m n/2}} \frac{ \Gamma\left( 1 + \frac{\beta}{2} \right)^n }{\Gamma\left( 1 + \frac{n\beta}{ 2}\right) \tilde \Gamma_n^\beta(p)} \frac{ \tilde \Gamma_{n-1}^\beta(p+\beta)}{ \Gamma\left( 1 + \frac{\beta}{2} \right)^{n-1} }\\
    & = n \frac{2^{\beta (m+1)(n-1)/2}}{2^{\beta m n/2}} \frac{ \Gamma\left( 1 + \frac{\beta}{2} \right)\tilde \Gamma_{n-1}^\beta(p+\beta) }{\Gamma\left( 1 + \frac{n\beta}{ 2}\right) \tilde \Gamma_n^\beta(p)}.
\end{align*}
Then we need to simplify
\begin{align*}
    \frac{ \tilde \Gamma_{n-1}^\beta(p+\beta) }{ \tilde \Gamma_n^\beta(p)} &= \frac{ \prod_{\ell = 1}^{n-1} \Gamma\left( \frac{\beta}{2} (m -n + 1) + \beta + \frac{\beta}{2} (\ell - 1) \right)}{\prod_{\ell = 1}^{n} \Gamma\left( \frac{\beta}{2} (m -n + 1) + \frac{\beta}{2} (\ell - 1) \right)}\\
    & = \frac{ \prod_{\ell = 1}^{n-1} \Gamma\left( \frac{\beta}{2} (m -n) + \frac{\beta}{2} (\ell + 2) \right)}{\prod_{\ell = 1}^{n} \Gamma\left( \frac{\beta}{2} (m -n) + \frac{\beta}{2} \ell \right)}\\
    & = \frac{\Gamma\left(\frac{\beta}{2}(m+1)  \right)}{\Gamma\left(\frac{\beta}{2}(m-n +1)  \right)\Gamma\left(\frac{\beta}{2}(m-n +2)  \right)}\\
    & = \frac{\Gamma\left(\frac{\beta}{2}(m+1)  \right)}{\Gamma\left(p \right)\Gamma\left( p + \frac{\beta}{2}\right)}. 
\end{align*}
This all gives
\begin{align*}
    Z_{n,m}^\beta =  n 2^{-p} \frac{\Gamma\left( 1 + \frac \beta 2 \right)}{\Gamma \left( 1 + \frac{n\beta}{2}\right)}\frac{\Gamma\left(\frac{\beta}{2}(m+1)  \right)}{\Gamma\left(p \right)\Gamma\left( p + \frac{\beta}{2}\right)}.
\end{align*}
Then to estimate, we use that $\Gamma(x + a) = \Gamma(x) x^a (1 + o(1))$ as $x \to \infty$ for $a$ fixed to write
\begin{align*}
    n\frac{\Gamma\left( 1 + \frac \beta 2 \right)}{\Gamma \left( 1 + \frac{n\beta}{2}\right)}\frac{\Gamma\left(\frac{\beta}{2}(m+1)  \right)}{\Gamma\left( p + \frac{\beta}{2}\right)} = \frac{\Gamma\left( 1 + \frac \beta 2 \right)}{\frac{\beta}{2} p^{\beta/2}} \frac{\Gamma(x+y)}{\Gamma(x)\Gamma(y)}(1 + o(1)),
\end{align*}
where $x = n\beta/2$ and $y = p$.  This is just the reciprocal of the Beta function with asymptotics
\begin{align*}
    \frac{\Gamma(x+y)}{\Gamma(x)\Gamma(y)} = \frac{1}{\sqrt{2 \pi}} \frac{(x+y)^{x + y - 1/2}}{x^{x -1/2} y^{y-1/2}} ( 1 + o(1))
\end{align*}
as $x,y \to \infty$, found using Stirling's formula.  Since $m = n/d - \sigma_n$ where $0 \leq \sigma_n < 1$ we can write
$x + y = \frac \beta 2 (m + 1) = \frac{\beta n}{2 d} + \gamma_n$, where $0 < \gamma_n \leq \frac \beta 2 \leq 1$. Therefore
\begin{align*}
    (x+y)^{x + y - 1/2} &= \left( \frac{\beta n} {2 d} \right)^{\frac{\beta n}{2 d} + \gamma_n - 1/2} \left( 1 + \frac{\gamma_n}{\frac{\beta n}{2 d}} \right)^{\frac{\beta n}{2 d} + \gamma_n - 1/2},\\
    x^{x -1/2} &= \left( \frac{\beta n} {2} \right)^{\frac{\beta n}{2} - 1/2},\\
    y^{y-1/2} &= \left( \frac{\beta n(1/d-1)} {2 } \right)^{\frac{\beta n(1/d-1)}{2 } + \gamma_n - 1/2} \left( 1 + \frac{\gamma_n}{\frac{\beta n(1/d-1)}{2 }} \right)^{\frac{\beta n(1/d-1)}{2} + \gamma_n - 1/2}.
\end{align*}
Since $\gamma_n$ is positive and bounded by unity, as $n \to \infty$
\begin{align*}
    (x+y)^{x + y - 1/2} & \leq (\E + o(1)) \left( \frac{\beta n} {2 d} \right)^{\frac{\beta n}{2 d} + \gamma_n - 1/2},\\
    y^{y-1/2} & \geq \left( \frac{\beta n(1/d-1)} {2 } \right)^{\frac{\beta n(1/d-1)}{2 } + \gamma_n - 1/2}.
\end{align*}
This gives
\begin{align*}
    \frac{\Gamma(x+y)}{\Gamma(x)\Gamma(y)} \leq \frac{\E + o(1)}{\sqrt{2 \pi(1-d)}} \sqrt{x} \left( \frac{1}{1-d} \right)^{x (1/d -1)} \left( \frac{1}{d} \right)^x.
\end{align*}
\end{proof}}

\begin{lemma} \label{l:powers}
For fixed $k \in \mathbb Z$, $ 0 < d <1$
\begin{align}
    m_{k,d,n}:= \frac{1}{n} \mathbb E \left[ \tr W_{n,\beta,d}^k \right] &\to \int \lambda^k \rho_{\mathrm{MP},d}(\lambda) \D \lambda =: m_{k,d}\notag\\
   m_{k,d,n} - m_{k,d} &= \begin{cases} O(n^{-1}) & k \geq 0,\label{eq:mean}\\
   O(n^{-1/2}) & k < 0, \end{cases}\\
    \frac{1}{n^2} \mathbb E (\tr W_{n,\beta,d}^k - n m_{k,d})^2 & = \begin{cases} O(n^{-2}) & k \geq 0,\\
   O(n^{-1}) & k < 0, \end{cases}, \label{eq:variance}
\end{align}
as $n \to \infty$.  If $d=1$ these estimates hold for $k \geq 0$.
\end{lemma}
\begin{proof}
The case of $k \geq 0$ is classical and implies weak convergence of the ESM to the Marchenko--Pastur law, see \cite[Section 3.1]{Bai2010}, for example.  For negative powers more work is required.  Recall the definition \eqref{eq:dos}
\begin{align*}
    \int f(\lambda) \mathbb E \mu_{n,\beta,d}( \D \lambda)  = \frac{1}{n} \mathbb E  \tr f(W_{n,\beta,d}) := \frac 1 n \mathbb E \left[ \sum_{\ell=1}^n f(\lambda_\ell)\right],
\end{align*}
for all $f \in C_b(\mathbb R^+)$.  We extend this definition to $f(\lambda) = \lambda^{-k}$, ~~ $k > 0$.  Introduce a continuous truncation of $\lambda^{-k}$:
\begin{align*}
    g(\lambda; \epsilon) = \begin{cases} \left(\frac{2}{\epsilon}\right)^k & 0 \leq \lambda \leq \frac{\epsilon}{2},\\
    \lambda^{-k} & \text{otherwise}. \end{cases}
\end{align*}
The monotone convergence theorem gives
\begin{align*}
    \int \lambda^{-k} \mathbb E \mu_{n,\beta,d}( \D \lambda) = \lim_{\epsilon \downarrow 0} \int g(\lambda) \mathbb E \mu_{n,\beta,d}( \D \lambda) = \lim_{\epsilon \downarrow 0} \frac{1}{n} \sum_{\ell=1}^n \mathbb E[g(\lambda_\ell;\epsilon)] =\frac{1}{n} \sum_{\ell=1}^n \mathbb E[\lambda_\ell^{-k}]. 
\end{align*}
The last term is finite for fixed $k$, provided $n$ is sufficiently large by Lemma~\ref{l:marginal_bound}.

For the sake of notation, set $g(\lambda) = g(\lambda;d_-)$.  Then, consider
\begin{align*}
    \left| \int \lambda^{-k} (\mathbb E \mu_{n,\beta,d}( \D \lambda) - \mu_{\mathrm{MP},d}(\D \lambda)) \right| = \left| \int \lambda^{-k} \mathbb E \mu_{n,\beta,d}( \D \lambda) - \int g(\lambda) \mu_{\mathrm{MP},d}(\D \lambda) \right|\\
     \leq  \underbrace{\left| \int (\lambda^{-k} - g(\lambda)) \mathbb E \mu_{n,\beta,d}( \D \lambda)\right|}_{I_1} +  
     \underbrace{\left| \int g(\lambda) (\mathbb E \mu_{n,\beta,d}( \D \lambda) -\mu_{\mathrm{MP},d}(\D \lambda)) \right|}_{I_2}.
\end{align*}
We estimate each of these terms separately.  First, we use that 
\begin{align}
    \left| \int (\lambda^{-k} - g(\lambda)) \mathbb E \mu_{n,\beta,d}( \D \lambda)\right| &\leq \int_0^{d_-/2} \lambda^{-k} \mathbb E \mu_{n,\beta,d}( \D \lambda) \notag\\
    &\leq \mathbb E\left[\lambda_n^{-k}(n,\beta,d)  \mathbbm{1}_{\{\lambda_n(n,\beta,d) \leq d_-/2 \}} \right]. \label{eq:smallest}
\end{align}
and show that this tends to zero exponentially.  But to establish the last inequality introduce
\begin{align*}
    h(\lambda; \epsilon) = \begin{cases} \left(\frac{2}{\epsilon}\right)^k & 0 \leq \lambda \leq \frac{\epsilon}{2},\\
    \lambda^{-k} & \frac{\epsilon}{2} < \lambda \leq \frac{d_-}{2},\\
    \left(\frac{2}{d_-}\right)^k - \frac{1}{\epsilon} \left(\frac{2}{d_-}\right)^k \left(\lambda - \frac{d_-}{2}\right)& \frac{d_-}{2} < \lambda \leq \frac{d_-}{2} + \epsilon,\\
    0 & \text{otherwise}.\end{cases}
\end{align*}
The dominated convergence theorem then provides
\begin{align*}
    \int_0^{d_-/2} \lambda^{-k} \mathbb E \mu_{n,\beta,d}( \D \lambda) &= \lim_{\epsilon \downarrow 0} \int h(\lambda;\epsilon) \mathbb E \mu_{n,\beta,d}( \D \lambda) \\
    &= \lim_{\epsilon \downarrow 0} \frac{1}{n} \sum_{\ell=1}^n \mathbb E [ h(\lambda_\ell;\epsilon)] \\
    &= \frac{1}{n} \sum_{\ell=1}^n \mathbb E \left[ \lambda_\ell^{-k}(n,\beta,d)  \mathbbm{1}_{\{\lambda_\ell(n,\beta,d) \leq d_-/2 \}} \right].
\end{align*}
And each term in the last sum is bounded by setting $\ell = n$.

To estimate the expectation \eqref{eq:smallest}, we use estimates on the marginal density $R(\lambda)$ for $\lambda_n$.  Specifically, \eqref{eq:deviation} implies that for any $d > 0$
\begin{align}\label{eq:uniform}
    \int_{\frac{2}{m \beta}}^{d_-/2} \lambda^{-k} R(\lambda) \D \lambda \leq C_{k,d} \E^{-n c_{k,d}},
\end{align}
for some constants $C_{k,d},c_{k,d}$ that do not depend on $n$.  And so, for this term we are left estimating
\begin{align*}
    \int_0^{\frac{2}{m \beta}}\lambda^{-k} R(\lambda) \D \lambda.
\end{align*}

Then we use \eqref{eq:growth}, introducing some constants $c_d, C_d > 0$ to estimate
\begin{align*}
    \mathbb E\left[\lambda_n^{-k}(n,\beta,d)  \mathbbm{1}_{\left\{\lambda_n(n,\beta,d) \leq \frac{2}{\beta m} \right\}} \right] 
    &\leq C_{d}  \E^{n c_d} (m \beta)^k \int_0^{1} \lambda^{p-k-1}  \E^{-\lambda} \frac{\D \lambda}{\Gamma(p)}\\
    & \leq C_{d}  \frac{\E^{n c_d} (m \beta)^k}{\Gamma(p)} \leq C_{k,d} \E^{-c_{k,d} n},
\end{align*}
for some $C_{k,d},c_{k,d} > 0$. Indeed, this converges to zero super exponentially.  From \eqref{eq:uniform} we obtain
\begin{align*}
\mathbb E\left[\lambda_n^{-k}(n,\beta,d)  \mathbbm{1}_{ \left\{\frac{2}{\beta m} \leq \lambda_n(n,\beta,d) \leq \frac{d_-}{2} \right\}} \right]\leq C_{k,d} \E^{-c_{k,d} n},
\end{align*}
and we may assume these constants are the same.  Therefore $I_1 \leq C_{k,d} \E^{-c_{k,d} n}$. To estimate $I_2$ we write
\begin{align*}
    I_2 \leq  \left(\int_0^\infty |g'(\lambda)| \D \lambda \right) d_{\mathrm{KS}}(\mathbb E\mu_{n,\beta,d}, \mu_{\mathrm{MP},d})
\end{align*}
Then the Kolmogorov--Smirnov distance is given by \cite[Theorem~8.10]{Bai2010}
\begin{align*}
    d_{\mathrm{KS}}(\mathbb E \mu_{n,\beta,d},\mu_{\mathrm{MP},d}) = O(n^{-1/2}), \quad n \to \infty.
\end{align*}
And therefore $I_2 = O(n^{-1/2})$.  Finally, to the variance estimate \eqref{eq:variance} for $k < 0$. From \cite[(4.16) and Remark~4.1]{Lytova2009}
\begin{align*}
    \mathrm{Var} \left( \int g(\lambda) \mu_{n,\beta,d} ( \D \lambda) \right) \leq C \|g'\|_\infty n^{-2}.
\end{align*}
We then have
\begin{align*}
    \mathrm{Var} \left( \int \lambda^{-k}  \mu_{n,\beta,d} ( \D \lambda) \right)^{1/2} &\leq \mathrm{Var} \left( \int g(\lambda) \mu_{n,\beta,d} ( \D \lambda) \right)^{1/2}\\
    &+ \mathrm{Var} \left( \int (g(\lambda)-\lambda^{-k}) \mu_{n,\beta,d} ( \D \lambda) \right)^{1/2}
\end{align*}
Then
\begin{align*}
    \mathrm{Var} \left( \int (g(\lambda)-\lambda^{-k}) \mu_{n,\beta,d} ( \D \lambda) \right)^{1/2} &\leq \frac{1}{n} \sum_{\ell=1}^n \mathbb E\left[ (g(\lambda_\ell) - \lambda_\ell^{-k})^2 \right]^{1/2} \\
    & + \int (\lambda^{-k} - g(\lambda)) \mathbb E\mu_{n,\beta,d}(\D \lambda).
\end{align*}
The last term here tends to zero at a exponential rate.  And because $\lambda^{-k} - g(\lambda)$ is a non-negative, monotonic function it suffices to estimate
\begin{align*}
    \mathbb E\left[(g(\lambda_\ell) - \lambda_\ell^{-k})^2\right] \leq \mathbb E\left[\lambda_n^{-2k}(n,\beta,d)  \mathbbm{1}_{ \left\{0 \leq \lambda_n(n,\beta,d) \leq \frac{d_-}{2} \right\}} \right]
\end{align*}
which vanishes again, at an exponential rate.  This establishes \eqref{eq:variance} with $m_{k,d,n}$ in place of of $m_{k,d}$.  Then \eqref{eq:variance} follows from \eqref{eq:mean} once one notes that
\begin{align*}
    \mathbb E\left[( \tr W^k_{n,\beta,d} - nm_{k,d})^2\right] - \mathbb E\left[( \tr W^k_{n,\beta,d} - nm_{k,d,n})^2 \right] = n^2\left( m_{k,d,n} - m_{k,d}\right)^2.
\end{align*}
\end{proof}
The proof of Lemma~\ref{l:powers} implies the following corollary, which complements the inequality \eqref{eq:largest}.
\begin{corollary}\label{c:invbound}
  For each $k = -1,-2,\ldots$ there exists $C_k> 0$, independent of $n$, such that
  \begin{align}\label{eq:invbound}
    \mathbb E\left[ \lambda_n^k\right] \leq C_k.
  \end{align}
\end{corollary}

The final results that we need from random matrix theory come from \cite[Corollary~1.8]{Guionnet2000}
\begin{theorem}\label{t:lip}
For any Lipschitz function $f: \mathbb R^+ \to \mathbb R$
\begin{align*}
\mathbb P\left( \left| \int f(\lambda) \mu_{n,\beta,d}(\D \lambda) - \int f(\lambda) \mathbb E \mu_{n,\beta,d}(\D \lambda) \right| \geq t \right) \leq C_{f,d} \E^{-n^2 t c_{f,d}},\quad t > 0,
\end{align*}    
for some constants $C_{f,d},c_{f,d}$.
\end{theorem}

\begin{corollary}\label{cor:dev}
Let $f$ be a continuous function on $(0,\infty)$, Lipschitz in a neighborhood of $[(1-\sqrt{d})^2,(1+\sqrt{d})^2]$, with at most polynomial growth at $0$ and at $\infty$, then \begin{align}\label{eq:lsd}
\mathbb P\left( \left| \int f(\lambda) \mu_{n,\beta,d}(\D \lambda) - \int f(\lambda) \mathbb E \mu_{n,\beta,d}(\D \lambda) \right| \geq t \right) \leq C_{f,d} \E^{-n t c_{f,d}},\quad 0 < t \leq 1,
\end{align}  
for some constants $C_{f,d},c_{f,d}>0$.
\end{corollary}
\begin{proof}
Let $[(1-\sqrt{d})^2-t,(1+\sqrt{d})^2+t] \subset (a,b)$, $a >0$ such that $f$ is Lipschitz on $[a,b]$.  Then define
\begin{align}\label{eq:tildef}
    \tilde f(x) = \begin{cases} f(x) & a \leq x \leq b\\
    f(a) & x < a\\
    f(b) & x > b.\end{cases}
\end{align}
And estimate
\begin{align*}
    \left| \int f(x) (\mu_{n,\beta,d}(\D \lambda) - \mathbb E \mu_{n,\beta,d}(\D \lambda)) \right|&\leq
    \left| \int (f(x)-\tilde f(x)) (\mu_{n,\beta,d}(\D \lambda) - \mathbb E \mu_{n,\beta,d}(\D \lambda))  \right|\\
    &+ \left| \int \tilde f(x) (\mu_{n,\beta,d}(\D \lambda) - \mathbb E \mu_{n,\beta,d}(\D \lambda)) \right|.
\end{align*}
By Theorem~\ref{t:lip}
\begin{align*}
    \mathbb P\left( \left| \int \tilde f(x) (\mu_{n,\beta,d}(\D \lambda) - \mathbb E \mu_{n,\beta,d}(\D \lambda)) \right| \geq t \right) \leq C_{\tilde f,d} \E^{-n^2 t c_{\tilde f,d}},
\end{align*}
Using Theorem~\ref{t:geb}
\begin{align}\label{eq:diff}
    \mathbb P \left( \int (f(x)-\tilde f(x)) \mu_{n,\beta,d}(\D \lambda) \neq 0\right) \leq C_0 \E^{-c_0 n}
\end{align}
By assumption, there exists $p,q >0$ such that
\begin{align*}
    f(x) \leq \begin{cases} C_1 x^{-p} & 0 \leq x \leq 1,\\
    C_2 x^q & x \geq 1. \end{cases}
\end{align*}
Following the proof of Lemma~\ref{l:powers}, we find that
\begin{align*}
    \int_0^a |f(x)-\tilde f(x)|  \mathbb E \mu_{n,\beta,d}(\D \lambda) \leq \frac 1 2 C_3 \E^{- c_3 n}.
\end{align*}
Similarly, using Theorem~\ref{t:geb} (with the same constants, for convenience)
\begin{align*}
    \int_b^\infty |f(x)-\tilde f(x)|  \mathbb E \mu_{n,\beta,d}(\D \lambda) \leq \frac 1 2 C_3 \E^{- c_3 n}.
\end{align*}
So, the last quantity to estimate is
\begin{align*}
  \mathbb P \left( \int_0^\infty |f(\lambda) - \tilde f(\lambda)| \mathbb E\mu_{n,\beta,d}(\D \lambda) \geq t \right) \quad t > 0.
\end{align*}
Then this probability is bounded above by $\mathbb P(C_3 \E^{-c_3 n} \geq t)$ and
\begin{align*}
    \mathbb P(C_3 \E^{-c_3 n} \geq t) = \begin{cases} 0 & n > \frac{1}{c_3} \log C_3,\\
    1 & n \leq \frac{1}{c_3} \log C_3. \end{cases}
\end{align*}
Suppose $t/C_3 < 1$.  Then $\mathbb P(C_3 \E^{-c_3 n} \geq t) \leq \mathbbm{1}_{\{ n/\log (C_3/t) < c_3 \}} \leq \E^{-n/\log (C_3/t)} \E^{1/c_3}$. Now $n/\log (C_3/t) \geq n \frac{t}{C_3}$ because $\frac{t}{C_3} \log \frac{C_3}{t} \leq 1 $ as $x \log x^{-1} \leq 1/\E$ for $0 < x \leq 1$.  Hence  $\mathbb P(C_3 \E^{-c_3 n} \geq t) \leq \E^{1/c_3} \E^{-n t/C_3}$.  However, if $t/C_3 \geq 1$, then clearly $\mathbb P(C_3 \E^{-c_3 n} \geq t) = 0$ so $\mathbb P(C_3 \E^{-c_3 n} \geq t) \leq \E^{1/c_3} \E^{-n t/C_3}$ for all $t > 0$. The corollary follows by applying Lemma~\ref{l:3facts}(2) twice.
\end{proof}

\begin{remark}
One might expect Corollary~\ref{cor:dev} to hold for all $t > 0$.  For this to indeed be true,  $f$ needs to be globally Lipschitz (i.e., Lipschitz on every compact subset of $(0,\infty)$) and the dependence of $C_{f,d}$ and $c_{f,d}$ in Theorem~\ref{t:lip} on $f$ needs to be known.
\end{remark}

\section{Proofs of the main theorems}\label{sec:proof}

\begin{proof}[Proof of Theorem~\ref{t:estimates}]
We first use invariance.  It follows that the errors  $\|\vec e_k\|_{W^{\ell}_{n,\beta,d}}$, $\vec e_k = \vec e_k(W,\vec b)$ realized in the CGA are invariant under unitary transformations, i.e. for $\tilde W = U WU^*$
\begin{align*}
    \|\vec e_k(W,\vec b)\|_{W^{\ell}} = \|\vec e_k(UWU^*,U\vec b)\|_{UW^{\ell}U^*}
\end{align*}
for any unitary matrix $U$.  This follows because if $p_k(\lambda)$ a polynomial of degree $k$ then (recall $\vec x_0 = 0$)
\begin{align*}
    \|p_k(W)\vec e_0(W,\vec b)\|_{W^{\ell}}^2 &= \|W^{\ell/2} p_k(W)\vec x\|_2^2 = \|U^*(UWU^*)^{\ell/2} p_k(UWU^*) U\vec x\|_2^2 \\
    &= \|p_k(\tilde W)\vec e_0(\tilde W, U\vec b)\|_{\tilde W^{\ell}}^2.
\end{align*}
And so, the mimimum over $ p_k \in\mathbb P_k^{(0)}$ must be same in both cases.  So, by invarince of $W_{n,\beta,d}$ it suffices to solve
\begin{align*}
    W_{n,\beta,d} \vec x = \vec b_0 = [1,0,\ldots,0]^T.
\end{align*}
We then recall formula \eqref{eq:pks} for $p_k^\dagger(\lambda)$ with $T_k = T_k(W_{n,\beta,d}, \vec b_0)$
\begin{align*}
    \|\vec e_k\|_{W^{\ell}_{n,\beta,d}}^2 = \sum_{j=1}^n \lambda_j^{\ell -2} p_k^\dagger(\lambda_j)^2 \omega_j = \int \lambda^{\ell -2} \frac{\det (T_k - \lambda I)^2}{ \det T_k^2} \nu_{n,\beta,d} (\D \lambda)
\end{align*}
where
\begin{align}\label{eq:nu}
    \nu_{n,\beta,d} = \sum_{j=1}^n \delta_{\lambda_j} \omega_j.
\end{align}
Here the distribution of $\boldsymbol \omega$ is parameterized by (see Appendix~\ref{a:wishart})
\begin{align*}
    \boldsymbol \omega = \frac{\boldsymbol \nu}{\|\boldsymbol \nu\|_1},
\end{align*}
where $\boldsymbol \nu$ is a vector of iid $\chi_\beta$-squared random variables.  The variable $\boldsymbol{\nu}$ is the square of the first components of the eigenvectors of $W_{n,\beta,d}$.  It is well-known that the eigenvalues and eigenvectors of $W_{n,\beta,d}$ are independent.  But, $T_k$ is dependent on both the eigenvalues and eigenvectors.

\begin{lemma}\label{l:1}
For $n > 0$,
\begin{align*}
    \mathbb E \left[ \left| \int \lambda^{\ell -2} \frac{\det (T_k - \lambda I)^2}{ \det T_k^2} \nu_{n,\beta,d} (\D \lambda) - \int \lambda^{\ell -2} {\det (T_k - \lambda I)^2} \nu_{n,\beta,d} (\D \lambda) \right| \right] \leq \frac{C}{\sqrt{n}}.
\end{align*}
\end{lemma}
\begin{proof}
We begin with a simple observation
\begin{align}\label{eq:coefs}
     \lambda_i^{\ell -2} {\det (T_k - \lambda_i I)^2} = \sum_{j=0}^{2k} t_j(\alpha_1,\ldots,\alpha_k,\beta_1,\ldots,\beta_{k-1}) \lambda_i^{j+\ell-2}.
\end{align}
By Lemma~\ref{l:chibound} it follows that for $q > 0$, $\mathbb E[t_j^q] \leq C_{j,q}$, where the bound is independent of $n$.  Similarly
\begin{align*}
    \mathbb E[\lambda_i^q] \leq C_q
\end{align*}
regardless of if $q$ is positive or negative, see \eqref{eq:largest} and \eqref{eq:invbound}.  Therefore, it suffices to show that
\begin{align*}
    \frac{1}{\det T_k^2} \overset{L^2}{\to} 1.
\end{align*}
Because of \eqref{eq:Hnb} we have
\begin{align*}
    \det T_k \overset{\text{dist.}}{=} \prod_{j = 1}^{n} \frac{\chi_{\beta(m-j+1)}^2}{\beta m}
\end{align*}
where these chi-distributed random variables are independent.  Then repeated use of the identity
\begin{align}\label{eq:product}
    \frac{1}{a}\frac{1}{b} -1 = \frac{1}{a} \left( \frac 1 b -1 \right) + \left( \frac 1 a - 1\right)
\end{align}
gives
\begin{align*}
    \frac{1}{\det T_k^2} -1 = \sum_{j=1}^k \left( \frac{\beta^2 m^2}{\chi^4_{\beta(m-j+1)}} - 1 \right) \left( \prod_{i=j+1}^{k} \frac{\beta^2 m^2}{\chi^4_{\beta(m-j)}}  \right).
\end{align*}
The first term tends to zero in any $L^p$ norm by Lemma~\ref{l:chibound} at a rate $n^{-1/2}$, and the second term is bounded uniformly in any $L^p$ norm.
\end{proof}
Next, we argue that while the measure is still random, we can replace the integrand with a deterministic one.
\begin{lemma}\label{l:2}
For $n > 0$,
\begin{align*}
     \mathbb E \left[ \left| \int \lambda^{\ell -2} {\det (T_k - \lambda I)^2} \nu_{n,\beta,d} (\D \lambda) - \int \lambda^{\ell -2} {\det (\mathbb T_k - \lambda I)^2} \nu_{n,\beta,d} (\D \lambda) \right| \right] \leq \frac{C}{\sqrt{n}}.
\end{align*}
\end{lemma}
\begin{proof}
Write $\det (\mathbb T_k - \lambda I)^2 = \sum_{j=0}^{2k} \tau_j \lambda^j$.  Using the notation of \eqref{eq:coefs}, it suffices to show that $|t_j - \tau_j| \to 0$ in $L^2$ at a rate $n^{-1/2}$.  Consider the product
\begin{align*}
    \prod_{j = 1}^k \chi_{\beta(n-d_j)}^{p_j} \chi_{\beta(m-s_j)}^{q_j}
\end{align*}
where $p_j,q_j \in \{0,1,2,3,4\}$ and $0 \leq d_j, s_j \leq k$ where the chi random variables are independent. Using \eqref{eq:product} we write
\begin{align*}
    P = \frac{\prod_{j = 1}^k \chi_{\beta(n-d_j)}^{p_j} \chi_{\beta(m-s_j)}^{q_j}}{\prod_{j = 1}^k (\beta d m)^{p_j/2}(\beta m)^{q_j/2}} - 1 
   & = \sum_{j=1}^k \left( \chi_j^{(1)} \chi_j^{(2)}- 1 \right) \prod_{i=j+1}^k \left( \chi_j^{(1)} \chi_j^{(2)}\right),
\end{align*}
where $\chi_{j}^{(1)} = \chi_{\beta(n-d_j)}^{p_j}/(\beta d m )^{p_j/2}, ~\chi_{j}^{(1)} = \chi_{\beta(n-s_j)}^{q_j}/(\beta m )^{q_j/2}$.  Then
\begin{align*}
    \mathbb E \left[ P^2 \right] \leq \sum_{j=1}^k \mathbb E \left[ \left(\chi_j^{(1)} \chi_j^{(2)} -1\right)^4  \right]^{1/2} \mathbb E \left[ \prod_{i=j+1}^k \left( \chi_j^{(1)} \chi_j^{(2)}\right)^2  \right]^{1/2}.
\end{align*}
Using \eqref{eq:product} and Lemma~\ref{l:chibound} it follows that $\mathbb E \left[ \left(\chi_j^{(1)} \chi_j^{(2)} -1\right)^4  \right]^{1/2} = O(n^{-2})$.  Then for $n > 0$
\begin{align}\label{eq:term-diff}
   \frac{1}{(\beta m)^{\frac{p+q}{2}}}\mathbb E \left[\left| \prod_{j = 1}^k \chi_{\beta(n-d_j)}^{p_j} \chi_{\beta(m-s_j)}^{q_j} - \prod_{j = 1}^k (\beta d m)^{p_j/2} (\beta m)^{q_j/2} \right|^2\right]^{1/2} = O\left( n^{-1/2}\right),
\end{align}
where $p = \sum_j p_j$ and $q = \sum_j q_j$.  This follows from again using Lemma~\ref{l:chibound}, which implies that all $L^p$ norms of $\chi_j^{(i)}$ are bounded as $n \to \infty$.  Then, one notes that the $L^2$ norm of $t_j - \tau_j$ can be bounded by a sum of terms of the form \eqref{eq:term-diff}.  This establishes this lemma.

\end{proof}

\begin{lemma}\label{l:3}
For $n > 1$,
\begin{align*}
     \mathbb E \left[ \left| \int \lambda^{\ell -2} {\det (\mathbb T_k - \lambda I)^2} \nu_{n,\beta,d} (\D \lambda) - \int \lambda^{\ell -2} {\det (\mathbb T_k - \lambda I)^2} \mu_{n,\beta,d} (\D \lambda) \right|   \right] \leq C\frac{ \log n}{\sqrt{n}}.
\end{align*}
\end{lemma}
\begin{proof}
Write $f(\lambda) = \lambda^{\ell -2} {\det (\mathbb T_k - \lambda I)^2}$ and integrate by parts
\begin{align*}
    I(f): = \int f(\lambda) &(\nu_{n,\beta,d} (\D \lambda) - \mu_{n,\beta,d} (\D \lambda)) =  \int_{\lambda_n}^{\lambda_1} f(\lambda) (\nu_{n,\beta,d} (\D \lambda) - \mu_{n,\beta,d} (\D \lambda))\\
    &=  \int_{\lambda_n}^{\lambda_1} f'(\lambda) F_{\beta,d}(x) \D \lambda
\end{align*}
where $F_{\beta,d}(x) = \mu_{n,\beta,d}((-\infty,x]) -\nu_{n,\beta,d}((-\infty,x]) $. Therefore
\begin{align*}
    |I(f)| \leq \left(\int_{\lambda_n}^{\lambda_1} |f'(\lambda)| \D \lambda \right) d_{\mathrm{KS}}(\mu_{n,\beta,d},\nu_{n,\beta,d}).
\end{align*}
Therefore
\begin{align*}
    \mathbb E \left[ |I(f)| \right] = \mathbb E \left[ \int_{\lambda_n}^{\lambda_1} |f'(\lambda)| \D \lambda \right] \mathbb E \left[d_{\mathrm{KS}}(\mu_{n,\beta,d},\nu_{n,\beta,d}) \right],
\end{align*}
by the independence of eigenvalues and eigenvectors ($d_{\mathrm{KS}}(\mu_{n,\beta,d},\nu_{n,\beta,d})$ is independent of the eigenvectors).  Then, we just note that there exists power $p,q \geq 0$ such that
\begin{align*}
    |f'(\lambda)| \leq C_k ( \lambda^{-p} + \lambda^q),
\end{align*}
and therefore
\begin{align*}
    \mathbb E \left[ \int_{\lambda_n}^{\lambda_1} |f'(\lambda)| \D \lambda \right]
\end{align*}
is bounded uniformly in $n$ by \eqref{eq:largest} and \eqref{eq:invbound}.  The lemma follows from Lemma~\ref{l:KSdist}.
\end{proof}

These three lemmas combined with Lemma~\ref{l:powers} establishes the Theorem~\ref{t:estimates}(1).

For the second part, we again establish a series of lemmas.

\begin{lemma}\label{l:4}
For $n \geq 0$
\begin{align*}
\mathbb P \left( \left| \int \lambda^{\ell -2} \frac{\det (T_k - \lambda I)^2}{ \det T_k^2} \nu_{n,\beta,d} (\D \lambda) - \int \lambda^{\ell -2} {\det (\mathbb T_k - \lambda I)^2} \nu_{n,\beta,d} (\D \lambda)  \right| \geq t \right) \\\leq C \E^{-c g(t) n}.
\end{align*}
for a non-decreasing function $g(t)$ that satisfies $g(t) > 0$ for $t > 0$, and for some constant $C > 0$.
\end{lemma}
\begin{proof}
For $t \geq 0$, let $\Lambda_d(C)$ be the event on which $C^{-1} \leq \lambda_n \leq \lambda_1 \leq C$ for $C > (1+ \sqrt{d})^{2}$ and $1/C < (1 - \sqrt{d})^{-2}$.  Then
\begin{align*}
    \mathbb P(\Lambda_d(C)) \geq 1 - 2 \E^{-n g_d(C)}
\end{align*}
where $g_d(C) > 0$.  This follows from \eqref{eq:deviation}.  Now, we make two elementary observations about
\begin{align}
     \lambda^{\ell -2} \frac{\det (T_k - \lambda I)^2}{\det T_k^2} = \sum_{j=0}^{2k} \tau_j \lambda^{j+\ell-2}.
\end{align}
Recall \eqref{eq:H} and it follows that $\tau_j = \tau_j(H_{n,\beta,d}/\sqrt{\beta m})$ is a Lipschitz function of the entries $(h_{ij})_{i\geq j}$ of $H_{n,\beta,d}/\sqrt{\beta m}$ in any closed $\epsilon$-neighborhood $0< \epsilon < 1$ of $\mathbb H_d$ in the max norm\footnote{The max norm gives the maximum entry, in modulus.} on lower-triangular matrices.  Let $L_{\epsilon,j}$ be the Lipschitz constant. The second observation is to let $Z_d(t)$ be the event where
\begin{align*}
    \max \left| \mathbb H_d - \frac{H_{n,\beta,d}}{\sqrt{\beta m}} \right| \leq |t|.
\end{align*}
By Lemma~\ref{l:chibound}, for $0 < t \leq \epsilon  \leq 1$,  $\mathbb P(Z_d(t)) \geq 1 - C_{k,d} \E^{-n c_{k,d}} $ for some constants $C_{k,d},c_{k,d} > 0$.  Therefore
\begin{align*}
    1 - &C_{k,c} \E^{- c_{k,d}n} - 2 \E^{-n g_d(C)} \leq \mathbb P(Z_d(t),\Lambda_d(C)) \\
    &\leq \mathbb P \left(   \sum_{j=0}^k |\tau_j(H_{n,\beta,d}/\sqrt{\beta m})- \tau_j (\mathbb H_d)| \leq |t|C^{2k}\sum_{j=0}^{2k} L_{t,j} \right)
\end{align*}
The lemma follows.
\end{proof}

\begin{lemma}\label{l:5}
For $n \geq 0$
\begin{align*}
\mathbb P \left( \left| \int \lambda^{\ell -2} {\det (\mathbb T_k - \lambda I)^2} \nu_{n,\beta,d} (\D \lambda) -\int \lambda^{\ell -2} {\det (\mathbb T_k - \lambda I)^2} \mu_{n,\beta,d} (\D \lambda) \right|   \geq t \right) \leq C \E^{-c g(t) n}.
\end{align*}
\end{lemma}
\begin{proof}
Recalling the notation $\Lambda_d(C)$ of the proof of the previous lemma, we then define the event
\begin{align*}
    K_d(t) = \left\{ d_{\mathrm{KS}}(\mu_{n,\beta,d},\nu_{n,\beta,d}) \geq t \right\}.
\end{align*}
Using the notation of \eqref{eq:KStail}
\begin{align*}
    1 - C_1 n e^{-c_1 n \beta t^2} -  C_2 n e^{-c_2 n \beta t} \leq \mathbb P\left( K_d(t)\right).
\end{align*}
Then
\begin{align*}
   1 - C_1 n e^{-c_1 n \beta t^2} -  C_2 n e^{-c_2 n \beta t} - 2 \E^{-n g_d(C)} \leq  \mathbb P \left( \Lambda_d(C), K_d(t)\right)
\end{align*}
and then we find for a constant $C_k > 0$
\begin{align*}
    \mathbb P \left( \Lambda_d(C), K_d(t)\right) \leq \mathbb P \left( \sup_{\lambda \in [C^{-1},C]} \left| \frac{\D}{\D \lambda} \lambda^{\ell -2} \det (\mathbb T_k - \lambda I)^2\right| d_{\mathrm{KS}}(\mu_{n,\beta,d},\nu_{n,\beta,d})  \leq C_k t \right).
\end{align*}
Therefore
\begin{align*}
    1 - C_1 n e^{-c_1 n \beta t^2} -  C_2 n e^{-c_2 n \beta t} - 2 \E^{-n g_d(C)} \\
    \leq \mathbb P \left( \left| \int \lambda^{\ell -2} {\det (\mathbb T_k - \lambda I)^2} \nu_{n,\beta,d} (\D \lambda) -\int \lambda^{\ell -2} {\det (\mathbb T_k - \lambda I)^2} \mu_{n,\beta,d} (\D \lambda) \right|   \leq C_k t \right)
\end{align*}
and this establishes the lemma.
\end{proof}
Applying Corollary~\ref{cor:dev}  establish along with these two lemmas establishes Theorem~\ref{t:estimates}(2).

For the case of $d = 1$ and $\ell \geq 2$ no inverse powers of $\lambda$ will be encountered in any integral.  So, the fact that Lemma~\ref{l:powers} applies only for $k \geq 0$ is not an issue.   Theorem \ref{t:geb} holds for $d= 1$, Theorem \ref{t:lip} indeed holds for $d = 1$ and Corollary~\ref{cor:dev} holds for $k = 1$ provided the function $f$ is Lipschitz continuous at $\lambda = 0$.  And Lemmas~\ref{l:1}, \ref{l:2}, \ref{l:3}, \ref{l:4} and \ref{l:5} hold for $d =1 $ provided $\ell \geq 2$.
\end{proof}

\begin{proof}[Proof of Theorem~\ref{t:ek}]
To evaluate 
\begin{align*}
    \mathfrak e^2_{\ell,k,d} : = \int \lambda^{\ell -2} \det(\mathbb T_{k,d} - \lambda I)^2 \mu_{\mathrm{MP},d}(\D \lambda)
\end{align*}
we make a simple change of variable $ \lambda =  \frac{d_+ - d_-}{2} x +  \frac{d_+ + d_-}{2} = 2 x \sqrt{d} + 1 + d$ so that
\begin{align}\label{eq:mfe}
    \mathfrak e_{\ell,k,d} : = \frac{2}{\pi} \int_{-1}^1 (2 x \sqrt{d} + 1 + d)^{\ell-3}  \det(\mathbb T_{k,d} - (2 x \sqrt{d} + 1 + d) I)^2 \sqrt{1-x^2}\, \D x.
\end{align}
Then examine \newcommand{\half}{\frac{1}{2}}
\begin{align*}
    \frac{1}{\sqrt{d}} \left( \mathbb T_{k,d} - (2 x \sqrt{d} + 1 + d) I \right)= \begin{bmatrix} - \sqrt{d} -  2x & 1 \\
    1 & -2x & \ddots \\
    & \ddots & \ddots & 1  \\
    & & 1 & -2x \end{bmatrix} =: D_{k,d}(-x).
\end{align*}
Next, we define $\det D_{0,d}(x) = 1$ and compute
\begin{align}\label{eq:recur}
  \begin{split}
    \det D_{1,d}(x) &= 2x  - \sqrt{d} ,\\
    \det D_{k+1,d}(x) &+  \det D_{k-1,d}(x) = 2x \det D_{k,d}(x), \quad k \geq 1.
  \end{split}
\end{align}

Note that \eqref{eq:recur} is the recurrence relation for the Chebyshev polynomials $T_n$ and $U_n$ of the first and second kinds.  We need some elementary properties of $T_n$ and $U_n$ (see, e.g. \cite{Mason2003}):
\begin{align*}
    U_n(\cos \theta) &= \frac{\sin (n+1) \theta}{\sin \theta}, \quad T_n(\cos \theta) = \cos n \theta,\\
    \frac{2}{\pi}&\int_{-1}^1 U_j(x) U_k(x) \sqrt{1-x^2}\, \D x = \delta_{jk},\\
    U_j(x) U_k(x) &= \sum_{\ell = 0}^{\min\{j,k\}} U_{|j-k| + 2 \ell}(x),\\
    \frac{1}{x+a} &= \frac{2}{\sqrt{a^2-1}} {\sum_{j=0}^\infty}' ( a - \sqrt{a^2 -1} )^j T_j(-x),
\end{align*}
where the $'$ denotes that the $j = 0$ term is halved. From the last equality it follows by differentiation that
\begin{align*}
    \frac{1}{(x+a)^2} &= \frac{2}{\sqrt{a^2-1}} \sum_{j=1}^\infty ( a - \sqrt{a^2 -1} )^j jU_{j-1}(-x).
\end{align*}
Recalling \eqref{eq:mfe} with $\ell = 1$ we have
\begin{align*}
    \frac{1}{(2x\sqrt{d}+d +1)^2} &= \frac{1}{\sqrt{d}(1-d)} \sum_{j=0}^\infty d^{j/2+1/2} (j+1) U_{j}(-x).
\end{align*}
Matching initial conditions for $D_{k,d}$ at $k = 1$ and $k = 2$ we obtain
\begin{align*}
    \det \left( \mathbb T_{k,d} - (2 x \sqrt{d} + 1 + d) I \right) = d^{k/2} \left[ U_k(-x) - \sqrt{d} U_{k-1}(-x) \right].
\end{align*}
Therefore
\begin{align*}
    &\frac{2}{\pi} \int_{-1}^1\frac{1}{(2x\sqrt{d}+d +1)^2}(U_k(-x) - \sqrt{d} U_{k-1}(-x))^2 \sqrt{1-x^2}\, \D x\\
    &= \frac{1}{\sqrt{d}(1-d)} \left( (2n +1) \sqrt{d}^{2n+1} + \sum_{k=0}^{n-1} \left[ (1+d)(2k+1) - 2d (2k+2) \right] \sqrt{d}^{2k+1} \right)
\end{align*}
Continuing,
\begin{align*}
    &\sum_{k=0}^{n-1} \left[ (1+d)(2k+1) - 2d (2k+2) \right] \sqrt{d}^{2k+1} =  \sqrt{d}\sum_{k=0}^{n-1} \left[ (1-d)(2k+1) - 2d \right] d^k\\
    &= \sqrt{d}(1-3d)\sum_{k=0}^{n-1}d^k + 2(1-d) d^{3/2}\sum_{k=1}^{n-1} k d^{k-1} \\
    & = \sqrt{d}(1-3d)\frac{1-d^n}{1-d}+ 2(1-d) d^{3/2}\frac{-n d^{n-1} (1-d) + 1 - d^n}{(1-d)^2}\\
    & = \sqrt{d}(1-3d)\frac{1-d^n}{1-d}+ 2 d^{3/2}\frac{n d^{n} - n d^{n-1} + 1 - d^n}{1-d}\\
    & = \sqrt{d}(1-d)\frac{1-d^n}{1-d}+ 2 d^{3/2}\frac{n d^{n} - n d^{n-1}}{1-d}\\
    & = \sqrt{d}(1-d^n)+ 2 d^{3/2}\frac{n d^{n-1}(d-1)}{1-d}\\
    & = \sqrt{d} - (2n+1) d^{n + 1/2}
\end{align*}
and this gives
\begin{align*}
    \mathfrak e^2_{1,k,d} = \frac{d^k}{1-d}.
\end{align*}
For $\mathfrak e_{2,k,d}$, we use $T_k(x) = \half U_k(x) - \half U_{k-2}(x)$ for $k \geq 1$ and $U_0(x) = T_0(x)$ to find
\begin{align*}
    \frac{1}{2x \sqrt{d} + 1 + d} & = \frac{1}{1-d}  {\sum_{j=0}^\infty} d^{j/2}\left[ U_j(-x) - U_{j-2}(-x) \right]\\
    & =  \sum_{j = 0}^\infty d^{j/2}  U_j(-x).
\end{align*}
Then
\begin{align*}
    &\frac{2}{\pi} \int_{-1}^1 \frac{1}{2x \sqrt{d} + 1 + d}(U_k(-x) - \sqrt{d} U_{k-1}(-x))^2 \sqrt{1-x^2}\, \D x\\
    & = d^k + (1+d) \sum_{j=0}^{k-1}d^j - 2d \sum_{j=0}^{k-1} d^{j}\\
    & = d^k + (1-d) \sum_{j=1}^{k-1}d^j\\
    & = d^k + 1-d^k\\
    & = 1.
\end{align*}
And this gives
\begin{align*}
    \mathfrak e^2_{2,k,d} = d^k.
\end{align*}
For $\mathfrak e_{3,k,d}$ we find
\begin{align*}
    &\frac{2}{\pi} \int_{-1}^1 (U_k(-x) - \sqrt{d} U_{k-1}(-x))^2 \sqrt{1-x^2}\, \D x = \begin{cases} 1 + d & k \geq 1,\\ 1 & k = 0. \end{cases}
\end{align*}
and this gives
\begin{align*}
    \mathfrak e^2_{3,k,d} = d^k \begin{cases} 1 + d & k \geq 1,\\ 1 & k = 0. \end{cases}
\end{align*}
Lastly, one can use the bound $|U_k(x)| \leq k$ to see that $\mathfrak e_{l,k,d} \to 0$ as $k \to \infty$ provided $0 < d < 1$.

\end{proof}

\appendix

\section{The eigenvalues and eigenvectors of Wishart matrices}\label{a:wishart}

Let $W = W_{n,\beta,d} = U \Lambda U^*$, $U^*U  = I$.  It is an important fact that the joint distribution of the vector
\begin{align*}
    \boldsymbol{\omega} = \begin{bmatrix}|U_{11}|^2 \\ \vdots \\ |U_{1n}|^2 \end{bmatrix}, \quad U = (U_{ij})_{1 \leq i, j \leq n}
\end{align*}
can be parameterized by
\begin{align}\label{eq:paramit}
   \boldsymbol{\omega} \overset{\text{dist.}}{=} \frac{\boldsymbol{\nu}}{\|\boldsymbol{\nu}\|}_1
\end{align}
where $\boldsymbol{\nu}$ is a vector of iid $\chi_\beta^2$ random variables. For the convenience of the reader we now derive \eqref{eq:paramit}.
\begin{definition}
$O(n)$ (resp., $U(n)$) denotes the group of $n \times n$ orthogonal (resp., unitary) matrices.
\end{definition}

We recall some general facts about Haar measure (see, e.g., \cite{Folland}).
\begin{theorem}
Let $G$ be a locally compact Hausdorff topological group.  Then $G$ has a left invariant measure $\mu$ (i.e., a left Haar measure) and an right invariant measure $\nu$ (i.e., a right Haar measure) on the $\sigma$-algebra generated by all open subsets of $G$.  The measures are unique up to a positive multiplicative constant.  
\end{theorem}

For a Borel set $S$, let $S^{-1}$ be the set of inverses of $S$.  Define
\begin{align*}
    \mu_{-1}(S) := \mu(S^{-1}).
\end{align*}
Then it is easy to see that $\mu_{-1}$ is a right Haar measure.  Thus by uniqueness,
\begin{align}\label{eq:propto}
\mu(S^{-1}) = k \nu(S)
\end{align}
for some $k > 0$.   Now, on the other hand, the left translate of a right invariant measure is still right invariant.  Thus, for all $g$, by uniqueness,
\begin{align}\label{eq:translate}
    \nu(g^{-1} S) = \Delta(g) \nu(S),
\end{align}
for some positive scaling factor $\Delta(g)$ --- the modular function.  $\Delta(g)$ is a continuous group homomorphism into the multiplicative group of positive numbers.  A group is called \underline{unimodular} if $\Delta(g) \equiv 1$. Clearly, it follows from \eqref{eq:translate} that $G$ is unimodular if and only if Haar measure is both left and right invariant, i.e., $\nu(S) = k' \mu(S)$, $k' > 0$.  There are many examples of unimodular groups:  most importantly for us, compact groups, (such as $O(n)$ or $U(n)$) are unimodular.

If $G$ is unimodular, it follows from \eqref{eq:propto} that
\begin{align}\label{eq:findk}
    \mu(S^{-1}) = k \nu(S) = kk' \mu(S).
\end{align}
Setting $S = G$ in \eqref{eq:findk} we find $kk' = 1$.  Hence
\begin{align}\label{eq:same}
    \mu(S^{-1}) = \mu(S).
\end{align}
In particular for $U(n)$, we see that
\begin{align}\label{eq:Un}
    \mu(S^*) = \mu(S)
\end{align}
and for $O(n)$
\begin{align}\label{eq:On}
    \mu(S^T) = \mu(S).
\end{align}

We proceed to show that \eqref{eq:paramit} holds if $U$ is distributed according to Haar measure on $O(n)$ or $U(n)$.  In order to construct Haar measure on $O(n)$ (or $U(n)$ resp.)  let $X$ be an $n \times n$ matrix of iid real (or complex, resp.) standard normal random variables.  In such a setting we say that $X$ belongs to the real (or complex, resp.) Ginibre ensemble.  Then the QR decomposition of $X$ gives
\begin{align*}
    X = QR, \quad Q^TQ = I ~~(\text{or } Q^*Q = I, \text{resp.}), \quad R \text{ is upper triangular}, \quad R_{ii} > 0.
\end{align*}
The $QR$ decomposition is unique if $X$ is non-singular.
Now the Ginibre ensemble is clearly invariant under multiplication on the left by a matrix $G \in O(n)$ (or $G \in U(n)$, resp.),
\begin{align*}
    G X = G Q R.
\end{align*}
So, the pair $(GQ, R)$ gives the QR decomposition of $GX \overset{\text{dist.}}{=} X$.  Thus $Q \overset{\text{dist.}}{=} GQ$.  Hence $X \mapsto Q$ induces Haar measure on $Q$.  But
\begin{align*}
    \begin{bmatrix} X_{11} \\ \vdots \\ X_{n1} \end{bmatrix} = \begin{bmatrix} Q_{11} \\ \vdots \\ Q_{n1} \end{bmatrix} R_{11}
\end{align*}
implying
\begin{align}\label{eq:desired1}
    \begin{bmatrix} Q_{11} \\ \vdots \\ Q_{n1} \end{bmatrix} = \frac{1}{\sqrt{|X_{11}|^2 + \cdots + |X_{n1}|^2}}\begin{bmatrix} X_{11} \\ \vdots \\ X_{n1} \end{bmatrix}.
\end{align}
By \eqref{eq:On} it follows that 
\begin{align}\label{eq:desired2}
\begin{bmatrix} Q_{11} & \cdots & Q_{1n} \end{bmatrix} \overset{\text{dist.}}{=} \frac{\begin{bmatrix} X_{11} & \cdots & X_{n1} \end{bmatrix}}{\sqrt{|X_{11}|^2 + \cdots + |X_{n1}|^2}}.
\end{align}

\begin{remark}
From \eqref{eq:desired2} we see that the components of the ``first'' eigenvector are proportional to independent $\chi_\beta$ variables.   But there is no ``first'' eigenvector: this should be the distribution for the components of any one eigenvector!  This follows by just reordering the eigenvalues. 
\end{remark}

To establish \eqref{eq:paramit}, we now show that the eigenvectors of $W = XX^*$ are Haar distributed on $U(n)$ when the entries of $X$ are iid standard complex normal random variables.   Here $X$ is $n \times m$, $m \geq n$.  Similarly, the eigenvectors of $W = XX^T$, $X$ is $n \times m$, $ m \geq n$, are Haar distributed on $O(n)$, where the entries of $X$ are iid standard (real) normal random variables.  

We follow the argument in \cite{Forrester2010}.

\paragraph{\underline{Step 1}} We consider the complex case, $\beta = 2$.  The case $\beta = 1$ is similar.  Apply the QR decomposition to $X^*$ to obtain $X^* = U_1 T$ where $U_1$ is $m \times n$ and $T$ is $n \times n$, $U_1^*U_1 = I_n$, $T$ is upper triangular, $T_{ii} > 0$.  

We use the following notation: if 
\begin{align*}
    \D Y = \begin{bmatrix} \D Y_{11} & \cdots & \D Y_{1n} \\
    \vdots & & \vdots\\
    \D Y_{n1} &  \cdots & \D Y_{nn} \end{bmatrix}
\end{align*}
is the matrix of differentials of $Y$, then $(\D Y)$ denotes the wedge product of independent entries of $Y$; e.g. if $Y$ is real symmetric then
\begin{align*}
    (\D Y) = \prod_{1 \leq i \leq j \leq n} \D Y_{ij}.
\end{align*}
Then one finds (see \cite[Proposition 3.2.5]{Forrester2010} that
\begin{align}\label{eq:dX}
    (\D X) = \prod_{j=1}^n T_{jj}^{2(n-j) + 1} (\D T) (U_1^* \D U_1).
\end{align}

\paragraph{\underline{Step 2}} From $W = T^*T$, we find (see \cite[Proposition 3.2.6]{Forrester2010})
\begin{align}\label{eq:dW}
    (\D W) = 2^m \prod_{j=1}^n T_{jj}^{2(m-j) + 1} (\D T).
\end{align}

\paragraph{\underline{Step 3}} Substituting \eqref{eq:dW} into \eqref{eq:dX} we find
\begin{align*}
    (\D X) & = 2^{-m} \prod_{j=1}^n T_{jj}^{2(n-j) +1 - 2(m-j) -1}(\D W) (U_1^* \D U_1)\\
    & = 2^{-m} (\det W)^{m-n}(\D W) (U_1^* \D U_1),
\end{align*}
and so
\begin{align*}
    \E^{- \tr X X^*} = \frac{\E^{- \tr W}}{2^m} (\det W)^{m-n}(\D W) (U_1^* \D U_1).
\end{align*}
Integrating out the independent variables $U_1$ we finally arrive at the pdf for $W$
\begin{align}\label{eq:pdfW}
    \frac{\E^{- \tr W}}{C_{n,m}} (\det W)^{m-n}(\D W),
\end{align}
for some normalization constant $C_{n,m} > 0$. 

\paragraph{\underline{Step 4}} Recomputation in the $\beta = 1$ case we find the general formula for the pdf of $W$ for $\beta = 1,2$, and some constant $C_{n,m,\beta}$
\begin{align}\label{eq:genform}
    \frac{1}{C_{n,m,\beta}} \E^{- \beta/2 \tr W} (\det W)^{\frac{\beta}{2}\left(m -n +1 - \frac{\beta}{2} \right)} (\D W).
\end{align}

\paragraph{\underline{Step 5}} . Now we use the standard computation for $(\D W)$ when it is Hermitian ($\beta =1, 2$): For the spectral decomposition ($\beta = 2$) $W = Q \Lambda Q^*$, we find (see \cite[(1.11)]{Forrester2010}) 
\begin{align}\label{eq:dW2}
    (\D W) = \prod_{1 \leq j < k \leq n} |\lambda_k - \lambda_j|^2 \prod_{j=1}^n \D \lambda_j (Q^* \D Q)
\end{align}
and for $\beta = 1$, $W = Q \Lambda Q^T$
\begin{align}\label{eq:dW1}
    (\D W) = \prod_{1 \leq j < k \leq n} |\lambda_k - \lambda_j| \prod_{j=1}^n \D \lambda_j (Q^T \D Q).
\end{align}
Inserting \eqref{eq:dW1}, \eqref{eq:dW2} into \eqref{eq:genform} we find the pdf of $W$
\begin{align}
    \frac{1}{C_{n,m\beta}} \E^{- \beta/2 \sum_j \lambda_j} \prod_{j=1}^n \lambda_j^{\frac{\beta}{2}\left(m -n +1 - \frac{\beta}{2} \right)} |V(\lambda)|^\beta  \prod_{j=1}^n \D \lambda_j (Q^T \D Q),
\end{align}
where $Q^* = Q^T$ for $\beta = 1$ and $V(\lambda)$ is the Vandermonde for $\lambda_1, \ldots,\lambda_n$.

Finally we see that the singular values of $X$ and the singular vectors for $X$ are independent.  As $Q^* \D Q$ is left (and hence, right) invariant, we see that $Q^* \D Q$ is Haar measure and hence the eigenvectors of $W = X^*X$ are Haar distributed.  Therefore \eqref{eq:paramit} follows.


\bibliographystyle{amsalpha}
\bibliography{references}

\end{document}